\documentclass[11pt,reqno]{amsproc}

\usepackage[margin=1.0in]{geometry}
\usepackage[english]{babel}
\usepackage[utf8]{inputenc}
\usepackage{ulem}
\usepackage{amsmath,amssymb,amsthm, comment, mathtools}
\usepackage{hyperref,mathrsfs,xcolor}
\usepackage{enumitem}

\usepackage{tikz}
\usetikzlibrary{patterns}
\usetikzlibrary{fpu}
\usetikzlibrary{plotmarks}
\usetikzlibrary{decorations.markings}
\usepackage{pgfplots}
\usetikzlibrary{intersections, pgfplots.fillbetween}

\usepackage{genealogytree}

\mathtoolsset{showonlyrefs}

\newtheorem{proposition}{Proposition}[section]

\newcommand{\R}{\mathbb{R}}

\newcommand{\x}{\mathbf{x}}

\newcommand{\ve}{\varepsilon}
\newcommand{\rmd}{{\rm d}}

\usepackage[mathscr]{euscript}

\newcommand{\eee}{equation}
\newcommand{\be}{\begin{\eee}}
\newcommand{\ee}{\end{\eee}}

\numberwithin{equation}{section}
\newtheorem{theorem}{Theorem}
\numberwithin{theorem}{section}

\numberwithin{prop}{section}
\newtheorem{cor}{Corollary}

\numberwithin{cor}{section}
\newtheorem{lemma}{Lemma}
\numberwithin{lemma}{section}
\theoremstyle{remark}
\newtheorem*{remark}{Remark}
\theoremstyle{definition}
\newtheorem{defi}{Definition}

\title{\vspace*{-30mm} A geometric characterization of steady laminar flow}

\author{Theodore D. Drivas}
\address{Department of Mathematics, Stony Brook University, Stony Brook, NY, 11790}
\email{tdrivas@math.stonybrook.edu}

\author{Marc Nualart}
\address{Instituto de Ciencias Matem\'{a}ticas, Consejo Superior de Investigaciones Cient\'{i}ficas, 28049 Madrid, Spain}
\email{m.nualart-batalla20@imperial.ac.uk}

\begin{document}
\maketitle
\vspace{-10mm}
\begin{abstract}
We study the steady states of the Euler equations on the periodic channel or annulus.  We show that if these flows are laminar (layered by  closed non-contractible streamlines which foliate the domain), then they must be either parallel or circular flows. We also show that a large subset of these shear flows are isolated from non-shear stationary states. For Poiseuille flow, $(v(y),0)=(y^2,0)$,  our result shows that all stationary solutions in a sufficiently small $C^2$ neighborhood are shear flows.
We then show that if $v(y)=y^n$ with $n\geq 1$, then in any $C^{n-}$ neighborhood, there exist smooth non-shear steady states, traveling waves, and quasiperiodic solutions of any number of non-commensurate frequencies.  This proves the rigidity near Poiseuille is sharp.
 Finally, we prove that on general compact doubly connected domains, laminar steady Euler flows with constant velocity on the boundary must also be either parallel or circular, and the domain a periodic channel or an annulus. This shows that laminar free boundary Euler solutions must have Euclidean symmetry.
\end{abstract}

\vspace{2mm}
\section{Introduction}
The Euler equations for a stationary, ideal incompressible fluid occupying a domain $M$ read
\begin{align}
u\cdot \nabla u &= -\nabla p\quad\  \text{in} \ M,\\
\nabla \cdot u &=0 \quad\quad \ \ \   \text{in} \ M,\\
 u\cdot \hat{n} &=0 \quad\quad \ \ \   \text{on} \ \partial M,
\end{align}
where $\hat{n}$ is normal to the solid boundary $\partial M$.
Stationary states generally come in infinite dimensional families and play an important role in understanding the long time evolution from non-stationary data \cite{bedrossian2015inviscid, choffrut2012local, danielski2024complex, dolce2022maximally,  drivas2023singularity, drivas2024twisting, drivas2022conjugate, onsager1949statistical, shnirelman1993lattice, shnirelman2013long, yudovich2000loss}. In two dimensions, the velocity can be expressed through a Hamiltonian $\psi$ called the streamfunction, via $u=\nabla^\perp \psi$.  In terms of the streamfunction and the vorticity (curl of the velocity) $\omega:=\nabla^\perp \cdot u= \Delta \psi$, steady Euler reads
\begin{align}
\{ \psi, \Delta \psi\} &=0 \quad\  \text{in} \ M,\\
\psi &=c_i  \quad    \text{on} \ \partial M_i,
\end{align}
where $\{f, g\}= \nabla^\perp f \cdot \nabla g$ is the Poisson bracket, $c_i$ are constants and $\partial M_i$ are connected components of the boundary. As such, steady Euler solution can be understood as  built out of a collection of curves (streamlines) along which both the streamfunction and vorticity are constant. This simultaneous relation is highly constraining, making these streamlines act effectively as a bundle of elastic bands, constantly straining to alleviate their tension by minimizing their length while maintaining their enclosed area.  As such, provided the domain $M$ permits, equilibrium configurations often have Euclidean symmetry (they are either straight lines, or perfect circles).  Not always must this be the case; there are some steady states on these domains that break the symmetry, where internal transition layers (accompanied by hyperbolic critical points) divide up  the  flow domain, serving as new effective boundaries.  In this work, we show this is the only possibility:
\begin{theorem}\label{thm:rigidity}
Let $M$ be the straight periodic channel  or the circular annulus. Let $u\in C^2(M)$ be a stationary solution of the Euler equations in $M$.  Suppose that $u$ is laminar in the sense that all its streamlines are non-contractible loops. Then, $u$ is a shear flow or a circular flow respectively.
\end{theorem}	

 Hamel and Nadirashvili  \cite{hamel2017shear, hamel2019liouville, HM23} proved the same conclusion under the assumption of \textit{non-vanishing} velocity.  This, indeed, ensures that the flows they consider are laminar. However, due to the assumed non-trivial velocity, the streamlines all act like elastic bands with positive tension.  By contrast, our result assumes on $u$ only the geometric assumption of being laminar, and allows for regions of slack (zero velocity) to be present in some of the streamlines. It is important to remark that the periodicity assumption cannot be removed; indeed there are examples of flows on the infinite strip which are laminar and non-shear \cite{de2024monotone, gui2024classification}!   In this sense, our theorem is sharp.  Our method of proof further allows open regions of zero velocity confined on non-contractible maximally connected level sets of the associated stream-function, see Thm. \ref{thm:levelsetrigidity} in \S \ref{zeroregion}.
Many rigidity results for 2D Euler have appeared in various contexts, see  \cite{CastroLear23, CastroLear24b, CastroLear24, CDG21, CDG22, coti2023stationary, elgindi2024classification, gui2024classification, huang2023rigidity, LinZeng, nualart2023zonal, Ruiz, WangZhan}.

We now give some intuition for our main result, Theorem \ref{thm:rigidity}. Let $M$ be the annulus and $\psi_0: M\to \mathbb{R}$ be a function without critical points in $M$. In particular, $u_0=\nabla^\perp \psi_0$ is a laminar flow.  Let $\mathcal{O}_{\psi_0}$ be the orbit of $\psi_0$ in the area preserving diffeomorphism group $\mathcal{D}_\mu(M)$:
\be
\mathcal{O}_{\psi_0}  = \{ \psi:M \to \mathbb{R} \ : \ \psi= \psi_0\circ \varphi, \ \text{for some} \ \varphi\in \mathcal{D}_\mu(M)\}.
\ee
Let $E[\psi]$ be the Dirichlet energy of a function $\psi: M\to \mathbb{R}$:
\be
E[\psi]= \frac{1}{2} \int_M |\nabla \psi |^2 dA.
\ee
It is well known that if $\psi_0$ extremizes the energy among all $\psi\in \mathcal{O}_{\psi_0}$, then $\psi_0$ is the streamfunction of a stationary solution \cite{arnold1966geometrie, arnold2009topological}. 
We now explain why this indicates $\psi_0$ should be radial.
We begin by expressing the energy by the coarea formula  (eliciting the analogy with constant tension elastic bands, with each streamline contributing an energy equal to its velocity-weighted length) as
\be
E[\psi]= \frac{1}{2} \int_{\psi^{-1}(M)} \rmd c \oint_{\{\psi = c\}} |\nabla \psi| \rmd \ell
\ee
where $\rmd \ell$ is the one-dimensional Hausdorff measure on the level curves on $\psi$.  Now, by Cauchy–Schwarz:
\be\label{ineq}
{\rm length}(\{\psi = c\}) \leq \mu_\psi(c) ^{1/2} \left(\oint_{\{\psi = c\}} |\nabla \psi| \rmd \ell\right)^{1/2}
\ee
where $\mu_\psi(c) :=  \oint_{\{\psi = c\}}\frac{ \rmd \ell}{ |\nabla \psi|}$ is the period of revolution of a particle on the level curve $\{\psi = c\}$.
Note that $\mu_\psi(c) $ is constant on the orbit $\psi\in \mathcal{O}_{\psi_0}$, since $\mu_\psi(c) = \frac{\rmd}{\rmd c} {\rm area}(\{\psi \leq c\})$ by the co-area formula, and areas of sublevel sets are preserved on the orbit $\mathcal{O}_{\psi_0}$. We have then
\be\label{energylowerbound}
E[\psi]\geq  \frac{1}{2} \int_{\psi^{-1}(M)} \frac{{\rm length}(\{\psi = c\})^2}{\mu_\psi(c) }  \rmd c.
\ee
Now, let $\psi_*\in \mathcal{O}_{\psi_0}$ be the unique radial rearrangement of $\psi_0$.   Since $\psi_*\in \mathcal{O}_{\psi_0}$, we have $\psi^{-1}(M)=\psi_*^{-1}(M)$ and  $\mu_\psi(c) =\mu_{\psi_*}(c)$.  Moreover, by the isoperimetric inequality, we have 
\be
{\rm length}(\{\psi = c\})\geq {\rm length}(\{\psi_*= c\}),
\ee
since the curves encircle regions of the same area. Thus
\be
E[\psi]\geq  \frac{1}{2} \int_{\psi^{-1}(M)} \frac{{\rm length}(\{\psi_* = c\})^2}{\mu_{\psi_*}(c) }  \rmd c=  E[\psi_*],
\ee
since the inequality \eqref{energylowerbound} is equality for radial $\psi=\psi_*$.   As such,  unless $\psi_0$ is radial,  the energy can be reduced by a suitable wiggling of the streamlines. This clearly illustrates why, in the course of settling in the minimum energy state, streamlines strive for minimal length.  See Figure \ref{figuremin}.

		\begin{figure}[h!] 
		\centering
			\includegraphics[width=0.65\textwidth]{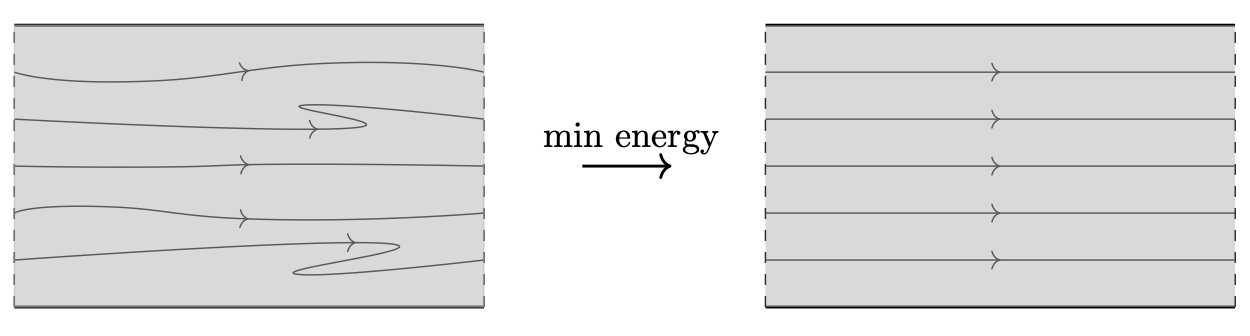} 
					\caption{Energy minimization procedure on the channel.} \label{figuremin}
	\end{figure}

 Although this argument serves as motivation for Theorem \ref{thm:rigidity}, it does not establish it for a couple reasons.  First, it assumes no stagnation points, otherwise the travel time will diverge.  Secondly, although it shows that the \textit{(global) minimal energy state} will always inherit the symmetry, it does not rule out the presence of some asymmetric local  minimums or saddle points (the presence of which is now ruled out by our main Theorem).   For this reason, here we take an entirely different viewpoint, one of over-constrained elliptic problems,  to prove Theorem \ref{thm:rigidity}.

Next, we list some consequences of our Theorem \ref{thm:rigidity} and its proof. Firstly, it shows that all steady states whose vorticity does not have any critical points are shear flows. Secondly, the presence of islands is still the only obstacle to Euclidean symmetry of the solution in perturbed domains, where recirculating eddies may form \cite{drivas2024islands}, if we further assume constant velocity on the boundaries. Indeed, a corollary of the proof of Theorem \ref{thm:rigidity} is
\begin{cor}\label{cor:rigidfreeboundary}
Let $M$ be an open bounded domain, doubly connected and such that $M\subset \mathbb{T}\times (0,1)$ or $M\subset \R^2$. Let $u\in C^2(\overline{M})$ be a steady laminar solution of the Euler equations in $M$ such that $|u|=a_i\geq 0$ on $\partial M_i$. Then, $u$ is a shear flow or a circular flow, respectively. 
\end{cor}
Such a problem arises for fluids with vacuum interfaces, see \cite{CDG22}. The above Corollary has two main implications. On the one hand, it shows that non-parallel steady Euler solutions on wrinkled domains either must have rich Lagrangian dynamics including islands of fluid, or the velocity field on the boundaries is non-constant. On the other hand, it implies that laminar free boundary fluid equilibria  must possess Euclidean symmetry.

Next, we show that certain of these steady laminar flows on the channel are robustly isolated from non-laminar ones:
\begin{theorem}\label{mainshearrigid}
Let $v(y)$ be such that $v(y_i)=0$ while $v''(y_i)\neq 0$ for finitely many $y_i\in (-1,1)$. Then, any steady state $u$ sufficiently close in $C^2$ to the shear flow $(v(y),0)$ is itself a shear flow.
\end{theorem}
Note that this theorem does not require the shear flow be non-stagnant, nor does it require any information on the vorticity along the stagnation curves. In the case of non-stagnant shear flows, the isolation takes place in much weaker topologies: $L^{2+}$ on vorticity (see Theorem \ref{rigiditynonstagnant}).  Theorem \ref{mainshearrigid} implies that Poiseuille flow  $v(y) = y^2$ is isolated in the $C^2$ topology from non-shear flows. In fact, any steady Euler solution $u$ such that $\Vert u - (y^2, 0)\Vert_{C^2(M)} < 2$ is necessarily a shear flow, since $\nabla\omega \neq 0$. This {regularity}/rigidity is sharp, as our next result shows

\begin{theorem} \label{flexthm}
Let $n\geq 1$ and consider $u_*(x,y) = (y^n,0)$ on $(x,y)\in M :=\mathbb{T}\times [-1,1]$. Then, for any $\ve>0$, there exist \textit{smooth} non-shear stationary states, time periodic and quasiperiodic of any number of non-commensurate frequencies, $u(t)$, such that $\|u-u_*\|_{C^{n-}(M)}<\ve$ for all time.
\end{theorem}
		\begin{figure}[h!] 
		\centering
			\includegraphics[width=0.9\textwidth]{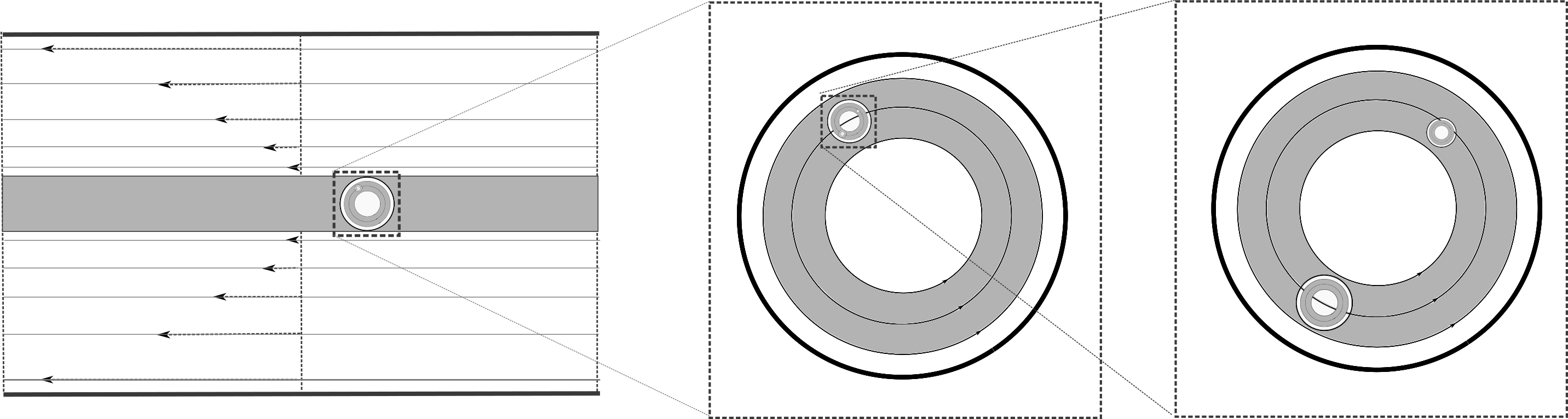} 
					\caption{A non-shear quasi-periodic solution nearby Poiseuille flow: radial vortex embedded in region of constant (zero) flow velocity, approximating Poiseuille flow.  A hierarchy of vortices periodically rotating embedded in regions of solid body rotation within each vortex.  Isochronal regions colored in gray.} \label{figure}
	\end{figure}
	
Of course the result holds for more general shear profiles; the only salient feature is the order of degeneracy near zeros (see Corollary \ref{cor:genflex}), which dictates the regularity of the H\"{o}lder spaces in which proximity is measured. The result is proved in Section \ref{flexsec}, but the idea of the construction is very simple, following along the same lines as Crouseilles--Faou \cite{crouseilles2013quasi}. First, we approximate the shear flow $v(y) = y^n$ by a zero velocity field in a small region. Then, we embed a compactly supported radial velocity field inside the approximating region. Since the background velocity is zero, this radial field constitutes a steady nonshear solution to Euler.  These circular vortices may in turn have enumerable other rotating  vortices living within them, in annular regions of solid body rotation (which have constant vorticity). See Figure \ref{figure}.

Two observations are in order. Firstly, for Couette flow $u_*=(y,0)$, Lin and Zeng proved \cite{LinZeng} that there are non-shear flows arbitrarily close to it in $H^{\frac52-}$ and that it is instead isolated in the $H^{\frac52+}$ topology. By Sobolov embedding, Couette is then isolated in the $C^{2,\frac12+}$ topology, but not in $C^{1,\frac12-}$, which is almost half a derivative better than our Theorem \ref{flexthm}.\footnote{This gap in the flexibility/rigidity threshold in H\"{o}lder regularity is due to the Sobolev embedding, since $H^{s+1}\subset C^s\subset H^s$, for all $s\geq 0$.} Secondly, we mention the result \cite[Theorem 1.3]{sinambela2023transition}, which shows the existence of non-shear steady solutions arbitrarily close to suitably chosen perturbations of strictly monotone shear flows in $C^{N,\frac12-}$, where $N+1$ denotes the maximal number of vanishing even derivatives of the background shear at the stagnation point ($N=1$ for Couette). 

The above two examples hint at the monotonicity and local parity of the background shear flow have a distinguished role in the regularity threshold of the flexibility/rigidity dichotomy, since these additional structures are able to improve the $C^{1-}$ regularity obtained by our Theorem \ref{flexthm}. Nevertheless, it is not clear if one can find analogous improvements for the odd  power laws $u_*(x,y)=(y^n,0)$, which are not strictly monotone for $n>1$ and for which Theorem \ref{flexthm} gives the best regularity for all $n\geq 2$, up to our knowledge. Although we do not prove it here for $u_*=(y^n,0)$ with $n>2$, we suspect that, as in the Poiseuille case, there exists an $\ve:=\ve(u_*, M)$ such that all stationary solutions $u$ satisfying $\|u-u_*\|_{C^n(M)}<\ve$ are shear flows. For other recent works on structures, both dynamic and static, nearby steady states see \cite{baldi2024nearly, CastroLear23, CastroLear24b, CastroLear24, enciso2023quasi, hassainia2023invariant, franzoi2024space}.

We now give an outline for the rest of the paper and the main ideas involved in the proofs. In Section \ref{sec:propertystream} we state and show basic properties of streamlines of steady laminar flows. We define suitable fluid sub-domains that give rise to semilinear elliptic equations for the stream-function, with $C^1$ non-linearities. If these fluid sub-domains have free boundaries, then we show that $\nabla\psi=0$ on them, thus obtaining over-determined elliptic equations.

Section \ref{sec:rigidoverdetchannel} studies these over-determined problems for domains that are subsets of the periodic channel and shows that the corresponding solutions (and the domains) must be invariant by horizontal translations. This is achieved using a sliding method together with local symmetry results from the theory of Continuous Steiner Symmetrization. 

Section \ref{sec:rigidoverdetannulus} considers the annular case and reaches the same conclusions; solutions must be now rotation-invariant. The arguments for the annular case, while adapted to our assumptions, are inspired by the rich literature on rigidity of steady Euler solutions in planar doubly connected sets.

Section \ref{sec:rigidfreeboundary} deals with the proof of Corollary \ref{cor:rigidfreeboundary}. To achieve this, we investigate free boundaries with constant velocity. There we establish an intermediate result concerning rigidity of solutions to the free-boundary Euler equations in the periodic strip. Its proof may be of independent interest in the theory of overdetermined boundary value problems in periodic domains and involves a moving plane method followed by a carefully devised sliding argument. 

Next, Section \ref{zeroregion} studies Theorem \ref{thm:rigidity} when open regions of zero velocity are allowed to exist. We finish with Section \ref{sec:mainshearrigid} and Section \ref{flexsec}, where we show that certain classes of shear flows are isolated from non-shear configurations, and we construct the existence of inviscid dynamical structures in lower regularity neighborhoods.

\section{Key properties of laminar streamlines and its fluid sub-domains}\label{sec:propertystream}
In this section we study the streamlines of laminar flows in $M$ and we construct semilinear elliptic equations in suitably chosen fluid sub-domains of $M$. The assumption in Theorem \ref{thm:rigidity} and Corollary \ref{cor:rigidfreeboundary} that all streamlines of the steady state are closed non-contractible loops has the following consequence: one can foliate $M$ by these streamlines\footnote{Note that our assumptions rule out the existence of open sets of zero velocity.  Indeed, if this were the case then the velocity field would admit contractible streamlines arbitrarily confined within these regions.  A more general case where zero regions are allowed is relegated to \S \ref{zeroregion}. }. Since $u=\nabla^\perp\psi$, the velocity field is tangent to the streamlines and, as a result, $u$ is a steady Euler solution on any domain whose boundaries are given by the streamlines of the flow. 
\subsection{Regular, Regular-Singular and Singular streamlines}
In order to define fluid sub-domains that are useful to our purposes, for each value $c$ in the range of $\psi:M\rightarrow \mathbb{R}$ we classify the streamlines associated to $c$. In what follows, any streamline $\Gamma_c \subseteq \lbrace \psi = c \rbrace$ is understood to be a maximally connected component of the set $\lbrace \psi^{-1}(c) \rbrace$ and it is assumed to be a closed non-contractible loop. We say that a streamline $\Gamma_c$ is
\begin{itemize}
\item a \emph{regular streamline} of $\psi$ if for all $q\in \Gamma_c$ there holds $\nabla\psi(q) \neq 0$.
\item a \emph{regular-singular streamline} of $\psi$ if there exists $p,q\in \Gamma_c$ such that $\nabla\psi(q)\neq 0$ and $\nabla\psi(p) = 0$.
\item a \emph{singular streamline} of $\psi$ if for all $p\in \Gamma_c$ we have $\nabla\psi(p) = 0$.
\end{itemize}
We remark here that this characterization is at the level of the streamlines, so that there may exist $c$ such that $\lbrace \psi^{-1}(c) \rbrace$ contains regular, singular and regular-singular streamlines. The following result shows that regular streamlines are dense in $M$.
\begin{lemma}\label{lemma:streamlines}
Regular streamlines are isolated from singular ones. Additionally, all points on any singular or regular-singular streamline are not isolated from regular streamlines. 
\end{lemma}

\begin{proof}We begin by noting that $\psi\in C^3(M)$ because $u\in C^2(M)$. If a regular streamline were not isolated from singular streamlines, we would be able to find a sequence of points $p_n$, with $\psi(p_n) = c_n \rightarrow c$ and $\nabla\psi(p_n) = 0$ such that $p_n\rightarrow p$, for some $p\in \Gamma_c$. By continuity, $\nabla\psi(p) = 0$. But this contradicts $\Gamma_c$ being a regular streamline.

To prove the second statement of the lemma, assume that regular-singular streamlines are isolated from regular ones. Now, let $\Gamma_c$ denote a regular-singular streamline and let $q_c\in \Gamma_c$ such that $\nabla\psi(q_c) \neq 0$. By continuity, $\nabla\psi(q)\neq 0$, for all $q\in B_{\delta_1}(q_c)$, for some $\delta>0$. Choosing $\delta>0$ sufficiently small, we see that all $q\in B_{\delta_1}(q_c)$ belong to regular-singular streamlines. Indeed, they cannot belong to singular streamlines because they have non-zero gradient and they cannot lie on regular streamlines since otherwise those would approximate the regular singular streamline. Now, the open ball $B_{\delta}(q_{c})$ is of positive area $\pi\delta^2$ and the velocity field $u= \nabla^\perp\psi$ has no stagnation points in  $B_{\delta}(q_{c})$, it moves points along their streamlines. We flow the ball by the field $u$. However, since all $q\in B_{\delta}(q_{c})$ belong to a regular-singular streamline, eventually the flowed region compresses to a set of zero area, a contradiction. 

Let now $p_c\in\Gamma_c$ be such that $\nabla\psi(p_c) = 0$ and assume that all $p\in B_\delta(p_c)$ do not lie on a regular streamline, for some $\delta>0$. Now, $\nabla\psi\equiv 0$ in $B_\delta(p)$ would contradict $\psi$ having non-contractible streamlines and thus there must be some $q\in B_{\delta}(p_c)$ such that $\nabla\psi(q)\neq 0$. By continuity, there exists some $\delta_1>0$ such that $\nabla\psi(p)\neq 0$, for all $p\in B_{\delta_1}(q)\subset B_{\delta}(p_c)$. As before,  $B_{\delta_1}(q)$ has non-zero area and can be flowed by the field $u=\nabla^\perp\psi$. However, all $p\in B_{\delta_1}(q)$ belong to regular-singular streamlines, (they cannot belong to singular streamlines because they have non-trivial gradient) so that eventually the ball compresses to a set of zero area. Thus, all points on regular-singular streamlines are not isolated from regular streamlines.
\end{proof}
\subsection{Construction of suitable fluid sub-domains}
The proof of the above lemma shows that in $M$ we must have regular streamlines. Otherwise all would necessarily be singular and the velocity field would necessarily be 0. Hence, to prove Theorem \ref{thm:rigidity}, for any regular streamline  $\Gamma_{c_0}\subseteq \lbrace \psi = c_0 \rbrace$, we shall see that the flow on $\Gamma_{c_0}$ is parallel on the periodic channel and circular on the annulus. To do so, for each regular streamline we devise an open bounded domain containing that streamline and where the stream-function satisfies a semilinear elliptic equation. As a matter of fact, these open bounded domains are doubly connected and have either parallel/circular boundaries or free boundaries where the stream-function enjoys over-determined conditions there.

To construct these fluid sub-domains we argue for the annulus since, by identification, the construction extends to the periodic channel. Now, by the Jordan curve Theorem, the smooth non-contractible curve $\Gamma_{c_0}$ separates the annulus into an open interior region denoted by $\Omega_{\Gamma_{c_0}}$ and an exterior region. Since $\nabla\psi \neq 0$ in $\Gamma_{c_0}$, we have that $\partial_{\hat{n}}\psi\neq 0$, where $\hat{n}$ denotes the exterior unit normal vector of $\partial\Omega_{\Gamma_{c_0}}=\Gamma_{c_0}$, and we may as well assume that $\partial_{\hat{n}}\psi <0$ throughout $\Gamma_{c_0}$, 
Next, we define
\begin{equation}
c_- := \inf \lbrace \psi(\Gamma_c) : \Gamma_c \text{ is a streamline},\,  \psi|_{\Gamma_c}=c \text{ and } \Omega_{\Gamma_c}\setminus \overline{\Omega_{\Gamma_{c_0}}} \neq \emptyset  \text{ has no singular streamlines} \rbrace
\end{equation}
and likewise
\begin{equation}
c_+ := \sup \lbrace \psi(\Gamma_c) : \Gamma_c \text{ is a  streamline},\, \psi|_{\Gamma_c}=c \text{ and }  \Omega_{\Gamma_{c_0}}\setminus \overline{\Omega_{\Gamma_c}} \neq \emptyset  \text{ has no singular streamlines} \rbrace.
\end{equation}
If $\partial_n\psi > 0$ in $\Gamma_{c_0}$, the definitions of $c_-$ and $c_+$ are modified accordingly. 

\begin{proposition}\label{prop:fluiddomain}
The supremum $c_+$ and infimum $c_-$ are achieved on streamlines $\Gamma_{c_+}$ and $\Gamma_{c_-}$, respectively. The open set $D:=\Omega_{\Gamma_{c_-}} \setminus \overline{\Omega_{\Gamma_{c_+}}}$ does not contain any singular streamline and $c_- < \psi < c_+$ in $D$. Additionally, if $\Gamma_{c_-}\subset M$ is interior, then it is a singular streamline and the same holds for $\Gamma_{c_+}$.
\end{proposition}

\begin{proof}
We shall argue for $c_-$, since the proof for the statements concerning $c_+$ is the same. Firstly, by assumption we have that $\Gamma_{c_0}$ is a regular streamline such that $\partial_{\hat{n}}\psi(p)< 0$, for all $p\in\Gamma_{c_0}$. For any such $p$, by continuity, we have that $\nabla\psi(q)\neq 0$, for all $q\in B_{\delta}(p)$, for some $\delta>0$. Hence, for all $s\in (0,\delta)$, we have that $\psi(p+s\hat{n}(p)) < \psi(p)$ and $\nabla\psi(p + s\hat{n}(p))\neq 0$. As such, for all $s\in (0,\delta)$, the streamline passing through $p+s\hat{n}(p)$ has value strictly smaller than $c_0$ and it is not singular, since $\nabla\psi(p+s\hat{n}(p))\neq 0$. This shows that we are taking an infimum over a non-empty set.

Now, let $\Gamma_{c_n}$ be a sequence of streamlines such that $\psi|_{\Gamma_{c_n}} = c_n>c$ and $\Omega_{\Gamma_{c_n}}\setminus \overline{\Omega_{\Gamma_{c_0}}}$ has no singular streamlines, with $c_n\rightarrow c$. We choose $p_n\in\Gamma_{c_n}$ and observe that the sequence of $p_n$ live in the compact space $\overline{M}\setminus \Omega_{\Gamma_{c_0}}$, so that, up to a subsequence, they converge to some $p\in\overline{M}$ such that, by continuity, $\psi(p) = c_-$. We then define $\Gamma_{c_-}$ as the closed streamline of $\psi^{-1}(c_-)$ that passes through $p$. By definition, it is such that $\Omega_{\Gamma_{c_-}}\setminus \overline{\Omega_{\Gamma_{c_0}}}$ has no singular streamlines. If $M$ is the periodic channel and $p\in\mathbb{T}\times \lbrace 1 \rbrace$, then $\Gamma_{c_-} = \mathbb{T}\times \lbrace 1 \rbrace$, while if $p$ is an interior point, then so is $\Gamma_{c_-}$. If $M$ is the annulus, the same conclusion holds, $\Gamma_{c_-}$ is the outer circular boundary.

The obvious modifications to the above argument also show the existence of a streamline $\Gamma_{c_+}$ where the supremum $c_+$ is achieved and $\Omega_{\Gamma_{c_0}}\setminus \overline{\Omega_{\Gamma_{c_+}}}$ has no singular streamlines. Since $\Gamma_{c_0}$ is itself a regular streamline, we conclude that $D:=\Omega_{\Gamma_{c_-}}\setminus \overline{\Omega_{\Gamma_{c_+}}}$ has no singular streamlines. Therefore, we must have $c_- < \psi < c_+$ in $D$. Otherwise, we would have some $q\in D$ either with $\psi(q)\leq c_-$ or $\psi(q) \geq c_+$ and that would imply first the existence of streamlines in $D$ with value $c_-$ or $c_+$ and then the existence of extrema in $D$. These extremal points would lie on a streamline, making all points in the streamline to be extremal, and thus rendering the streamline to be singular.

We now show that if $\Gamma_{c_-}$ lies in $M$ then it must be singular. Assume otherwise, namely that there exists some $q\in \Gamma_{c_-}$ such that $\nabla\psi(q)\neq 0$, that is, $\partial_{\hat{n}} \psi(q)\neq 0$. Now, if $\partial_{\hat{n}}\psi(q)>0$, there exists streamlines $\Gamma_c$ in $D$ with $\psi(\Gamma_c) < c_-$, thus contradicting $c_- < \psi < c_+$ in $D$. On the other hand, if $\partial_{\hat{n}}\psi(q) < 0$, by continuity we can likewise find non-singular streamlines $\Gamma_c$ such that $\psi|_{\Gamma_c} = c < c_-$ and $\Omega_{\Gamma_{c}}\setminus \overline{\Omega_{\Gamma_{c_0}}}$ has no singular streamlines, a contradiction with the fact that $c_-$ is the infimum. Therefore, it must be $\nabla\psi\equiv 0$ on $\Gamma_{c_-}$, it is a singular streamline. The same holds for $\Gamma_{c_+}$. The proof is concluded.
\end{proof}
\begin{figure}[h!]
\centering
\includegraphics[width=0.8\textwidth]{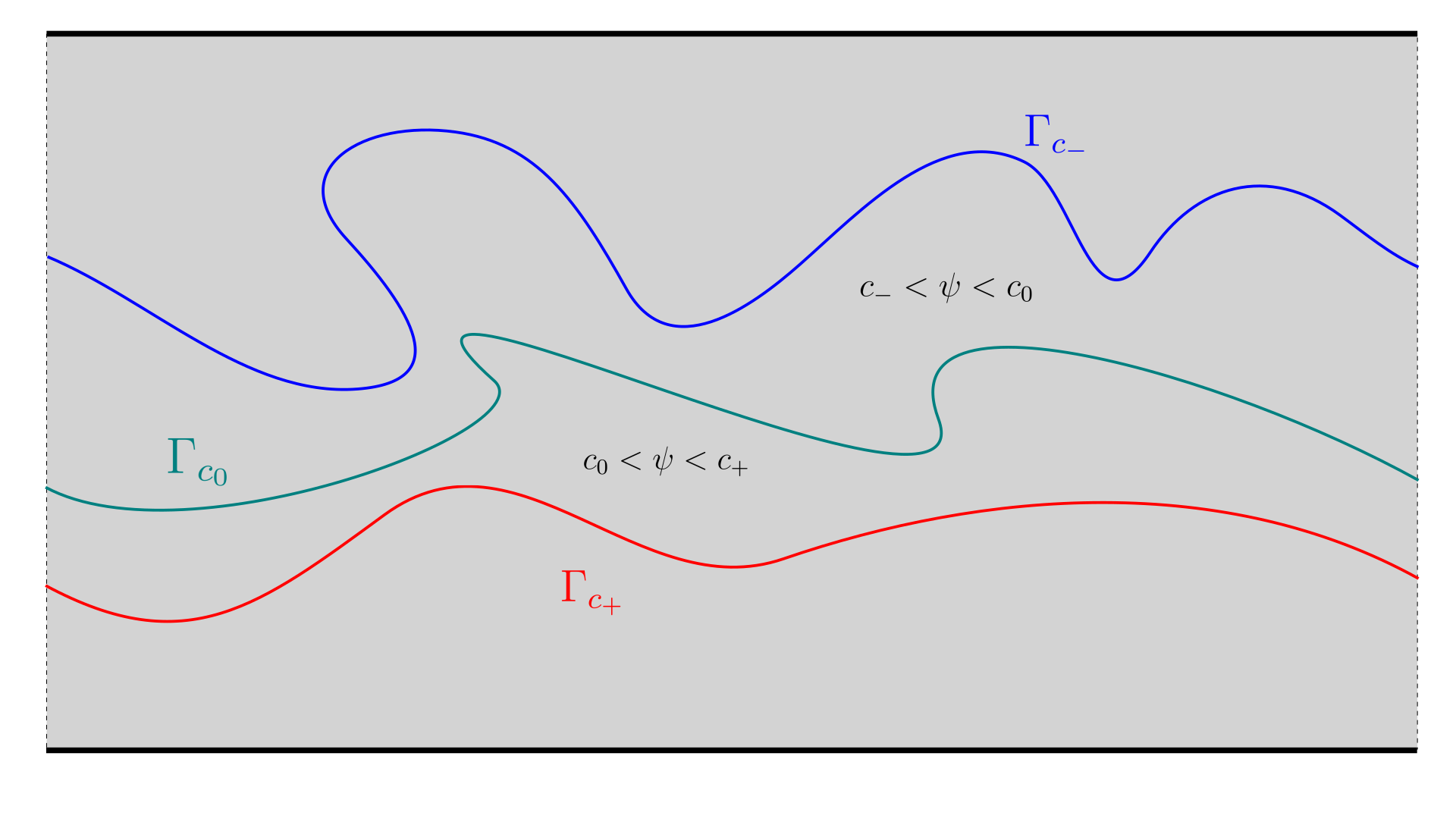}
\caption{A typical fluid sub-domain. The regular streamline $\textcolor{teal}{\Gamma_{c_0}}$ lives in the region bounded by the singular streamlines $\textcolor{blue}{\Gamma_{c_-}}$ and $\textcolor{red}{\Gamma_{c_+}}$. By construction, $c_- < \psi < c_+$ in $D$. }
\end{figure}
\subsection{A semilinear elliptic equation for the streamfunction}
Once we have found a fluid sub-domain $D$ containing only  regular and regular-singular streamlines, the next step is to show that $\psi$ satisfies a semi-linear elliptic equation there. Without loss of generality, we may assume that $D=\Omega_{\Gamma_0}\setminus \overline{\Omega_{\Gamma_1}}$, where $\Gamma_0$ and $\Gamma_1$ are streamlines of $\psi$.
\begin{lemma}\label{lemma:mainellipticeq}
Let $D=\Omega_{\Gamma_{0}}\setminus\overline{\Omega_{\Gamma_1}}$. Let $\psi$ be a steady Euler solution in $D$ such that all streamlines are closes non-contractible curves and they are not singular. Then, there exists $f\in C^1(0,1)\cap C[0,1]$ such that 
\begin{equation}\label{eq:ellipsi}
\begin{cases}
\Delta\psi = f(\psi), \quad 0<\psi<1 & \text{ in }D, \\
\psi = 0, & \text{ in } \Gamma_0, \\
\psi = 1, & \text{ in } \Gamma_1,
\end{cases}
\end{equation}
\end{lemma}
\begin{proof}
We begin by showing that $\psi$ must attain different values in each connected component of $\partial D$, namely $\Gamma_0$ and $\Gamma_1$. Indeed, if $\psi$ were to attain the same value, then the global maximum or minimum of $\psi$ in $\overline{D}$ is attained in $D$. If both maxima and minima are attained in the boundary, then $\psi$ would be constant in $D$ and thus contain singular streamlines in $D$. Assume by simplicity that the maximum is attained in $D$. Now, by hypothesis this maximal value $c$ must be achieved along a closed non-contractible streamline $\Gamma_c$. Since it is the maximum in $D$, it must have $\nabla\psi = 0$ at all points where it is achieved, hence $\nabla\psi|_{\Gamma_c}\equiv 0$. Thus, $\Gamma_c$ would be a singular streamline in $D$, a contradiction.

Hence, up to addition and multiplication by constants (which do not alter the streamline topology), we assume that $\psi=0$ in $\Gamma_0$ and $\psi=1$ in $\Gamma_1$. Let $c\in [0,1]$, which is the range of $\psi$ in $\overline{D}$ and let $p\in \overline{D}$ such that $\psi(p) = c$. We define $f(c) = \Delta\psi(p)$. We shall see that $f$ is well-defined, it does not depend on the choice of $p$. To do so, consider $c\in [0,1]$ and $\Gamma_c$ the unique associated streamline. Its uniqueness stems from the fact that  two or more distinct streamlines attaining the same value imply the existence of global extrema in the region bounded by these streamlines, and there are no singular streamlines in $D$. From the Euler equations, we see that $\nabla\Delta\psi $ is parallel to $\nabla\psi$ and thus, if $\Gamma_c$ is a regular streamline, then $\Delta\psi$ is automatically constant on $\Gamma_c$ and $f(c)$ is well defined. 

If $\Gamma_c$ is a regular-singular streamline, assume there exists $p_1,p_2\in \Gamma_c$ such that $\Delta\psi(p_1) = \omega(p_1) < \omega(p_2) = \Delta\psi(p_2)$, and let $\epsilon := \frac{\omega(p_2) - \omega(p_1)}{4}$. By continuity, there exists $\delta>0$ such that $|\omega(q) - \omega(p_i)| < \epsilon$, for all $q\in B_{\delta}(p_i)$, for $i=1,2$. Since points on regular-singular stream lines are not isolated from points on regular streamlines, there exists $q_i\in B_{\delta}(p_i)$ such that $\psi(q_i) = c_i<c$ and $q_i\in \Gamma_{c_i}$ is a regular streamline. 

If $c_1 = c_2$, then $\Gamma_{c_1}=\Gamma_{c_2}$, since otherwise we would have local extrema and thus a singular streamline within the fluid domain. Hence,  $\omega(q_1) = f(c_1) = f(c_2) = \omega(q_2)$, it is constant on the regular streamline $\Gamma_{c_1}=\Gamma_{c_2}$. One one hand, since $q_1\in B_{\delta}(p_1)$, $\omega(q_1)$ is at most $\frac{\omega(p_2) - \omega(p_1)}{4}$ away from $\omega(p_1)$, while on the other hand, since $q_2\in B_\delta(p_2)$, we have that $\omega(q_2)=\omega(q_1)$ is also $\frac{\omega(p_2) - \omega(p_1)}{4}$ away from $\omega(p_2)$. This is clearly  impossible. If, say, $c >c_1 > c_2$, then by the intermediate value theorem there exists some $q_2^* \in B_{\delta}(p_2)$ such that $\psi(q_2^*) = \psi(q_1) = c_1$, namely there is some point of $\Gamma_{c_1}$ in $B_{\delta}(p_2)$. Since $\omega = \omega(q_1)$ on $\Gamma_c$, we have that $\omega(q_2^*)=\omega(q_1)$ and $q_2^*\in B_\delta(p_2)$, a contradiction with the continuity of $\omega$. Therefore, $\omega$ is constant in $\Gamma_c$. The case $c=0$ or $c=1$ is treated similarly; $f(i)$ is well defined if $\Gamma_i$ is a regular or regular-singular streamline singular streamline. If it is a singular streamline, since any point on a singular streamline is not isolated from points lying on regular streamlines, see Lemma \ref{lemma:streamlines}, by continuity of $\omega$ and following the argument for regular-singular streamlines we conclude that $f(i)$ is well-defined as well.

The next step in the proof is to show that $f$ is continuous. To do so, let $c_n\in [0,1]$, with $c_n\rightarrow c$, for some $c\in[0,1]$. For each $c_n\in [0,1]$ we denote $\Gamma_{c_n}$ the unique streamline in $D$ where $\psi|_{\Gamma_{c_n}} = c_n$ and let $p_n\in \Gamma_{c_n}$. Up to a subsequence, we have that $p_n\rightarrow p$, for some $p\in \overline{D}$. By continuity of $\psi$, there holds $\psi(p) = c$. Moreover, by continuity of $\omega = \Delta\psi$, we have that
\begin{equation}
f(c_n) = \Delta\psi(p_n) \rightarrow \Delta\psi(p) = f(c),
\end{equation}
thus concluding that $f:[0,1]\rightarrow \R$ is continuous.

The last step in the proof is to show that actually $f$ is of class $C^1(0,1)$. To achieve this, for all $c\in(0,1)$ we shall obtain a $C^1$ expression for $f$ in $(c-\delta_1,c+\delta_2)$, for some $\delta_1,\delta_2>0$ small enough. Let then $c\in(0,1)$, $\Gamma_c$ its unique associated streamline in $D$ and $q\in \Gamma_c\subset D$ such that $\nabla\psi(q)\neq 0$. This is possible because all streamlines in $D$ are not singular. We next define the normal flow $\sigma:(-\epsilon, \epsilon) \rightarrow D$ by
\begin{equation}
\begin{cases}
\sigma'(t) = \nabla\psi(\sigma(t)), & t\in (-\epsilon, \epsilon), \\
\sigma(0) = q,
\end{cases}
\end{equation}
for some $\epsilon>0$. Define also $g=\psi\circ\sigma:(-\epsilon,\epsilon)\rightarrow (0,1)$ and observe that $g'(0) = |\nabla\psi(q)|^2 >0$. Hence, for some (maybe smaller) $\epsilon>0$ and for suitable $\delta_1>0$ and $\delta_2>0$, we have that $g:(-\epsilon, \epsilon) \rightarrow (c-\delta_1, c+\delta_2)$ is a $C^1$ diffeomorphism. Moreover, we can write
\begin{equation}
f(c) = \Delta\psi(\sigma(g^{-1}(c))),
\end{equation}
for all $c\in (c-\delta_1, c+\delta_2)$. Since $\psi\in C^3$, we see that $f$ in $(c-\delta_1, c+\delta_2)$ is given by a $C^1$ function. This finishes the proof of the Lemma.
\end{proof}

\begin{remark}
As one can see from the proof, the existence of points of vanishing gradient in $D$ does not prevent the existence of a  function $f$ mapping the stream-values to their vorticities in a smooth manner. Indeed, as long as any streamline contains points with non-trivial gradient, one can define a locally smooth normal flow map there and show the smoothness of $f$.
\end{remark}
\subsection{Over-determined boundary value problems}
According to Proposition \ref{prop:fluiddomain} there are only three types of sub-domains $D$ in $M$ that we shall consider. All of them are diffeomorphic to the annulus and on each of them there are no singular streamlines, so that $\psi$ must attain different values on each connected component of the boundary. Since we are interested in the topology of the streamlines, we shall now add and multiply by non-zero constants the stream-function on each fluid sub-domain. These operations do not change the streamline topology and simplify the analysis. Hence, we assume $\psi=0$ on the outermost boundary and $\psi=1$ in the innermost boundary. We denote $\partial M_{\mathrm{top}}$ and $\partial M_{\mathrm{bot}}$ the top and bottom boundary of $M$. In the case of the periodic channel, these correspond to the upper and lower horizontal boundaries, while in the case of the annulus these correspond to the outer and inner circular boundaries. The three types of domain we have are the following.
\begin{enumerate}
\item The set $D=M$. This occurs when $\psi$ has no singular streamlines in $M$. Hence, we have from Lemma \ref{lemma:mainellipticeq} that $\psi$ satisfies
\begin{equation}\label{eq:elliM}
\begin{cases}
\Delta\psi = f(\psi), \quad  0<\psi<1, & \text{in }M, \\
\psi = 0, & \text{on }\partial M_{\mathrm{top}}, \\
\psi = 1, & \text{on }\partial M_{\mathrm{bot}},
\end{cases}
\end{equation}
for some $f\in C[0,1]\cap C^1(0,1)$.
\item The set $D = \mathbb{T}\times(0,1) \setminus \overline{\Omega_{\Gamma_1}}$, for some singular streamline $\Gamma_1$. From Lemma \ref{lemma:mainellipticeq} we see that $\psi$ satisfies
\begin{equation}\label{eq:ellionefree}
\begin{cases}
\Delta\psi = f(\psi), \quad  0<\psi<1, & \text{in }M \setminus \overline{\Omega_{\Gamma_1}}, \\
\psi(x,1) = 0, & \text{on }\partial M_{\mathrm{top}}, \\
\psi = 1, \quad \nabla\psi = 0, & \text{on }\Gamma_1,
\end{cases}
\end{equation}
for some $f\in C[0,1]\cap C^1(0,1)$ such that. Here we mention that the setting where $D=\Omega_{\Gamma}\setminus \overline{\Omega_{\partial M_{\mathrm{bot}}}}$ can be transformed to $D = M \setminus \overline{\Omega_{\Gamma}}$ and \eqref{eq:ellionefree} by applying suitable transformations for which the Laplacian is invariant and that do not change the streamline topology.
\item The set $D = \Omega_{\Gamma_0} \setminus \overline{\Omega_{\Gamma_1}}$, for some singular streamlines $\Gamma_0$ and $\Gamma_1$. From Lemma \ref{lemma:mainellipticeq} we see that $\psi$ satisfies
\begin{equation}\label{eq:ellitwofree}
\begin{cases}
\Delta\psi = f(\psi), \quad  0<\psi<1, & \text{ in }\Omega_{\Gamma_0} \setminus \overline{\Omega_{\Gamma_1}}, \\
\psi = 0, \quad \nabla\psi = 0, & \text{ on }\Gamma_0, \\
\psi = 1, \quad \nabla\psi = 0, & \text{ on }\Gamma_1,
\end{cases}
\end{equation}
for some $f\in C[0,1]\cap C^1(0,1)$.
\end{enumerate}
To study symmetry and invariant properties of the solution $\psi$ to any of the above problems, we shall determine the regularity of $\partial D$. While we know that regular streamlines are $C^2$ curves (we can parametrize them by curves whose tangent vector is proportional to $\nabla^\perp\psi\neq 0$), and regular-singular streamlines are also $C^2$ locally around its regular points, in general we do not know how regular the singular streamlines are, since $\nabla\psi\equiv 0$ there. Nevertheless, in certain cases we are able to show they are $C^2$ regular curves. This is the purpose of the following result.

\begin{lemma}\label{lemma:streamlineregularity}
Let $D=\Omega_{\Gamma_{0}}\setminus\overline{\Omega_{\Gamma_1}}$. Let $f\in C^1(0,1)\cap C[0,1]$ and $\psi\in C^3(\overline{D})$ be a solution to
\begin{equation}\label{eq:ellipsibis}
\begin{cases}
\Delta\psi = f(\psi), \quad 0<\psi<1 & \text{ in }D, \\
\psi = 0, & \text{ on } \Gamma_0, \\
\psi = 1, & \text{ on } \Gamma_1,
\end{cases}
\end{equation}
If $\nabla\psi = 0 $ on $\Gamma_i$ and $f(i)\neq 0$, then $\Gamma_i$ is a $C^2$ curve.
\end{lemma}

\begin{proof}
If $\nabla\psi = 0$ on $\Gamma_i$, we cannot directly parametrize $\Gamma_i$ as $\lbrace \psi = i \rbrace$ using the Implicit Function Theorem on the equation $\psi(x,y)=i$. However, since $f\in C([0,1])$ and $\psi\in C^3(\overline{D})$, there holds $\Delta\psi = f(i)\neq 0$ on $\Gamma_i$. That is, for the usual Cartesian coordinates, $\partial_x^2\psi$ and $\partial_y^2\psi$ cannot vanish simultaneously on $\Gamma_i$. Hence, we can split $\Gamma_i = U_x \cup U_y$, where $U_x = \lbrace p \in \Gamma_i: \partial_x^2 \psi (p) \neq 0 \rbrace$ and $U_y$ is defined analogously. They are both relatively open sets of $\Gamma_i$, at least one of them is non-empty and on all their points we can define a local $C^2$ curve parametrizing $\Gamma_i$. Indeed, for any $p\in U_y$, say, we have $\nabla\psi(p) = 0$, in particular $\partial_y\psi(p)= 0$, while $\nabla\partial_y\psi(p)\neq 0$. Hence, by the Implicit Function Theorem, the set $\lbrace \partial_y\psi = 0 \rbrace$ is locally given by a unique $C^2$ curve $\gamma_p(t)$ such that $\gamma_p(0) = p$ and $\partial_y\psi(\gamma_p(t))=0$, for all $t$ sufficiently small. Now, since $\nabla\psi=0$ on $\Gamma_i$, by the uniqueness of the Implicit Function Theorem, we have that all points $q\in\Gamma_i$ sufficiently close to $p$ must lie in the $C^2$ curve $\gamma_p(t)$. Thus, $\Gamma_i$ can be locally parametrized by a $C^2$ curve.
\end{proof}

The above Lemma allows us to determine the sign of $f(i)$ whenever $\nabla\psi = 0$ on $\Gamma_i$.

\begin{lemma}\label{lemma:signf}
Let $\psi\in C^3(\overline{D})$ be a solution to \eqref{eq:ellipsibis}. If $\nabla\psi = 0$ on $\Gamma_i$, then $(-1)^{i+1}f(i)\leq 0$.
\end{lemma}

\begin{proof}
Assume $\nabla\psi=0$ on $\Gamma_1$ and $f(1) > 0$. From Lemma \ref{lemma:streamlineregularity}, $\Gamma_1$ is a $C^1$ curve and let $\hat{n}$ denote its exterior unit normal vector with respect to $D$. Since $\nabla\psi = 0$ on $\Gamma_1$, we have $\partial_{\hat{\tau}}^2 = \partial_{\hat{n}\hat{\tau}}\psi = 0$ on $\Gamma_1$ and thus $\Delta\psi = \partial_{\hat{n}}^2\psi = f(1)>0$. Hence, for $p\in\Gamma_1$ and $\delta>0$ small enough, there holds $p-\delta\hat{n}(p)\in D$ and $\psi(p-\delta\hat{n}(p)) >\psi(p)=1$, a contradiction. The proof for $\Gamma_0$ is the same, we omit the details.
\end{proof}

To conclude that $u$ is a shear flow if $M$ is the periodic channel, and a radial flow if $M$ is the annulus, we shall see that solutions $\psi$ to \eqref{eq:elliM}, \eqref{eq:ellionefree} and \eqref{eq:ellitwofree} are functions of one variable. Both \eqref{eq:ellionefree} and \eqref{eq:ellitwofree} are over-determined boundary value problems, and we show in the next section that solutions to these problems must necessarily be one-dimensional. Instead, \eqref{eq:elliM} is not a-priori an overdetermined problem, but it is posed already on $M$, the boundaries of which are already one dimensional (either horizontal straight lines in the channel or circular in the annulus). We shall show in the next section that this alone already gives invariance of the solutions with respect to the symmetries of the domain $M$, namely translational invariance in the periodic channel and rotational invariance in the annulus.

\section{Over-determined elliptic equations in the periodic channel}\label{sec:rigidoverdetchannel}
In this section we assume $M$ is the periodic channel and prove that the solutions to the overdetermined boundary value problems \eqref{eq:elliM}, \eqref{eq:ellionefree} and \eqref{eq:ellitwofree} must be one-dimensional. 
\subsection{The periodic channel}
We study here \eqref{eq:elliM}, which takes place when there are no singular streamlines in the periodic channel $\mathbb{T}\times(0,1)$. In particular, the domain where the equation is satisfied already has straight horizontal boundaries, and this fact is enough to constrain all solutions $\psi$ of \eqref{eq:elliM} to be $x$-translational invariant, that is, one-dimensional.
\begin{proposition}\label{prop:rigiditychannel}
Let $\psi$ be a solution of \eqref{eq:elliM}. Then, $\psi=\psi(y)$.
\end{proposition}
\begin{proof}
We shall use the sliding method. Take $\xi=(\xi_1,\xi_2)$, with $\xi_2>0$ and $\tau>0$. We define the comparison function
\begin{equation}
w^{\tau}(\x) = \psi(\x+\tau\xi) - \psi(\x)
\end{equation}
and denote
\begin{equation}
J:= \lbrace \tau\in (0, 1/\xi_2):, \quad w^\tau(\x) <0, \text{  for all } \x \in D^\tau \rbrace,
\end{equation}
where further
\begin{equation}
D^\tau := \lbrace \x\in D\,: \x+\tau\xi \in D \rbrace = \mathbb{T}\times (0, 1-\tau\xi_2)
\end{equation}
Since $\psi(x,0)=1$ and $\psi(x,1) = 0$ for all $x\in\mathbb{T}$, by continuity there holds $w^\tau(\x)<0$ in $D^\tau$, for $\tau\xi_2$ sufficiently close to 1 and we see that $J$ is non-empty. Next, we define $\tau_0 := \inf J$ and we claim that $\tau_0=0$. Otherwise, if $\tau_0>0$ then by continuity of $w^\tau$ and by definition of infimum, there exists some $\x_0\in \overline{D^{\tau_0}}$ such that $w^{\tau_0}(\x_0) = 0$. Given that $0<\psi<1$ in $D$, we see that $w^{\tau_0}(\x) < 0$, for all $\x\in \partial D^{\tau_0}$, since either $\x=(x,0)$ or $\x+\tau_0\xi=(x+\tau_0\xi_1,1)$. Therefore, both $\x_0$ and $\x_0+\tau_0\xi$ are interior points of $D^{\tau_0}$, where $f$ is $C^1$ and thus we can conclude from the maximum principle that $w^\tau \equiv 0$, initially in a small ball centred at $\x_0$ and then by unique continuation (recall $f\in C^1$ in the interior of $D^{\tau_0}$) and finally to the whole of $\overline{D^{\tau_0}}$, thus contradicting $w^{\tau_0}<0$ there. As a result, $\tau_0 = 0$. Finally, we take $\xi_2\rightarrow 0$ so that now $w^\tau(\x)\leq 0$ in $D$, for all $\tau\geq 0$. This requires $\psi$ to be invariant by arbitrary horizontal translations, namely $\psi=\psi(y)$ alone. 
\end{proof}

\subsection{Over-determined elliptic equations with free boundaries}
We now turn to the over-determined elliptic problems \eqref{eq:ellionefree} and \eqref{eq:ellitwofree}. In general, we cannot directly apply the sliding method used in Proposition \ref{prop:rigiditychannel} here because the domain $D$ may not be invariant under $x$-translations. Instead, since $\nabla\psi=0$ on the free boundaries, we choose to extend the stream-function $\psi$ from $D=\Omega_{\Gamma_0}\setminus \overline{\Omega_{\Gamma_1}}$ to the periodic channel $M=\mathbb{T}\times(0,1)$ through its boundary values. Namely, we extend $\psi \equiv 1$ in the region of $M$ bounded above by $\Gamma_1$ and $\psi \equiv 0$ in the region of $M$ bounded below by $\Gamma_0$. This extension produces a $C^1$ function that weakly satisfies a modified semilinear elliptic equation, now in a domain with horizontal straight boundaries. 

However, the non-linearity is no longer $C^1$ in the interior of this larger domain, and the sliding method is still not an option. Nevertheless, the periodic channel now enjoys symmetry and, once $\psi$ and $M$ are evenly reflected to $\mathbb{T}\times(-1,0)$, the semilinear elliptic equation falls within the applicability of the theory of local symmetry through the Continuous Steiner Symmetrization, a technique developed by Brock in \cite{Brock} to ascertain the symmetry of positive weak solutions to elliptic problems in domains with symmetries. Here we introduce the concept of local symmetry of functions, and we refer the interested reader to \cite{Brock, Brock2} for a complete account on the Continuous Steiner Symmetrization and its usability on variational problems and elliptic equations.
\begin{defi}
Let $\psi\geq 0$ and $\psi\in C^1(\lbrace 0 < \psi < \sup \psi \rbrace)$. Suppose this last set is open. Let $\x_1=(x_1,y_1)$ such that
\begin{equation*}
0 < \psi(\x_1) < \sup \psi, \quad \partial_y\psi(\x_1)>0,
\end{equation*}
and let $\x_2=(x_1,y_2)$ be the unique point such that
\begin{equation*}
y_2>y_1, \quad \psi(\x_1) = \psi(\x_2) < \psi(x_1,y), \quad \forall y\in (y_1,y_2).
\end{equation*}
We say that \emph{$\psi$ is locally symmetric in the direction $y$} if
\begin{equation*}
\partial_x\psi (\x_1) = \partial_x\psi(\x_2), \quad \partial_y\psi(\x_1) = - \partial_y\psi(\x_2).
\end{equation*}
\end{defi}
Locally symmetric functions admit a decomposition of their support,
\begin{equation*}
\lbrace 0 < \psi < \sup \psi \rbrace = \bigcup_{k\geq 1}\left( U_k^1 \cup U_k^2 \right) \bigcup S,
\end{equation*}
where $U_k^1$ is a maximally connected component of $\lbrace 0 < \psi < \sup \psi \rbrace \cap \lbrace \partial_y\psi > 0 \rbrace$, $U_k^2$ denotes its reflection about the hyperplane $\lbrace y = y_k \rbrace$, for some $y_k\in \R$ and $\partial_y\psi = 0 $ in $S$. In other word, for any $k\geq 1$ and any $\x=(x,y)\in U_k^1$, there holds
\begin{equation}
\psi(x,y) = \psi(x,2y_k - y) < \psi(x,z), \quad \forall z\in (y, 2y_k - y).
\end{equation}
The main goal of this section is to first show that appropriate extensions and transformations of solutions to \eqref{eq:ellionefree} and \eqref{eq:ellitwofree} are locally symmetric, and then upgrade this local symmetry to global symmetry of the original solutions. To do that, we first need a preliminary technical lemma.
\begin{lemma}\label{lemma:technical}
Let $D=\Omega_{\Gamma_0}\setminus\overline{\Omega_{\Gamma_1}}$ and $\psi$ be a solution of 
\begin{equation}\label{eq:psielliptic}
\begin{cases}
\Delta\psi = f(\psi), \quad 0<\psi<1 & \text{ in }D, \\
\psi = 0, & \text{ in } \Gamma_0, \\
\psi = 1, & \text{ in } \Gamma_1,
\end{cases}
\end{equation}
for some $f\in C[0,1]\cap C^1(0,1)$. Further assume that $D$ contains no singular streamlines of $\psi$ and $\partial_y\psi<0$ on all regular streamlines of $\psi$ in $D$. Then, all regular-singular streamlines contain points where $\partial_y\psi < 0$.
\end{lemma}

\begin{proof}
Firstly, by density of regular streamlines in $D$, we have $\partial_y\psi \leq 0$ in $D$. Let $\Gamma_c$ be a regular-singular streamline and assume towards a contradiction that $\partial_y\psi|_{\Gamma_c}\equiv 0$. Since it is a regular-singular streamline, there exists some $p\in \Gamma_c$ such that $\nabla\psi(p)\neq 0$, namely $\partial_x\psi(p)\neq 0$. By continuity, there exists some open ball $B_\delta(p)$ such that $\partial_x\psi\neq 0$ in $B_\delta(p)$. Next, we define $\gamma_c(s)=(\gamma_c^1(s), \gamma_c^2(s))$ as the solution to 
\begin{equation}\label{eq:gammapsi}
\begin{cases}
\dot{\gamma}_c(s) = \nabla^\perp\psi(\gamma_c(s)), & \\
\gamma_c(0) = p,
\end{cases}
\end{equation} 
which is well-defined for $|s|$ sufficiently small, there exists some $s_0>0$ such that $\gamma_c(s)\in B_\delta(p)$, for all $|s|< s_0$.  By definition of $\gamma_c$, there holds $\psi(\gamma_c(s)) =c$ since
\begin{equation}
0=\dot{\gamma}_c^1(s)\partial_x\psi(\gamma_c(s)) + \dot{\gamma}_c^2(s)\partial_y\psi(\gamma_c(s)),
\end{equation}
namely, $\gamma_c(s)$ gives a local parametrization of $\Gamma_c$ in a neighbourhood of $p$. Moreover, by continuity we see that $\partial_x\psi(\gamma_c(s))\neq 0$, for all $|s| < s_0$. Since $\partial_y\psi\equiv 0$ in $\Gamma_c$ and together with \eqref{eq:gammapsi}, we deduce that $\dot{\gamma}_c^1(s) = 0$ and $\dot{\gamma}_c^2(s) = \partial_x\psi(\gamma_c(s))\neq 0$, for all $|s|< s_0$. In particular, there exists some $\epsilon>0$ such that
\begin{equation}
\lbrace p + t(0,1) : t\in (-\epsilon, \epsilon) \rbrace \subset \Gamma_c,
\end{equation}
that is, the streamline $\Gamma_c$ contains a completely vertical segment.

By taking $\epsilon$ sufficiently small, we may assume that $B_\epsilon(p)\subset B_\delta(p)$ and let $p_0 = p + \frac{\epsilon}{2}(0,1)\in B_\epsilon(p)$, it is such that $B_{\frac{\epsilon}{4}}(p_0) \subset B_\epsilon(p)$. Assume $p=(p_x,p_y)$ and let $\pi(\x)$ denote the reflection of $\x\in B_{\frac{\epsilon}{4}}(p_0)$ with respect to the hyperplane $\lbrace y = p_y \rbrace$. Define also the comparison function 
\begin{equation}
w(\x) = \psi(\pi(\x)) - \psi(\x)
\end{equation}
in $B_{\frac{\epsilon}{4}}(p_0)$ and observe that, since $\partial_y\psi \leq 0$, there holds $w\geq 0$. Moreover, since $\psi$ satisfies the elliptic equation \eqref{eq:psielliptic} in $D$, we see that
\begin{equation}
\Delta w(\x) + g(\x)w(\x) = 0,
\end{equation}
for some $g\in L^\infty(B_{\frac{\epsilon}{4}}(p_0))$. Crucially, $p_0$ is an interior point of $B_{\frac{\epsilon}{4}}(p_0)$ and $w(p_0)=0$, since both $p_0\in \Gamma_c$ and $\pi(p_0) = p - \frac{\epsilon}{2}(0,1)\in \Gamma_c$. By the maximum principle, we conclude that $w\equiv 0$ in $B_{\frac{\epsilon}{4}}(p_0)$.

By density, there exists some $q\in B_{\frac{\epsilon}{4}}(p_0)$ belonging to a regular streamline. By hypothesis, $\partial_y\psi < 0$ in $q$ and by continuity in a small neighbourhood of $q$. Now, since $w \equiv 0$ in $B_{\frac{\epsilon}{4}}(p_0)$, $\partial_y\psi \leq 0$ and $\partial_y\psi <0$ on $q=(q_x,q_y)$ and also on $\pi(q) = (q_x,2p_y-q_y)$ because $\pi(q)$ belongs to the same regular streamline of $q$, we obtain
\begin{equation}
0 = \psi(\pi(q)) - \psi(q) = \int_{p_y-q_y}^{q_y-p_y}\partial_y\psi(p + s(0,1)) \rmd s < 0,
\end{equation}
thus reaching a contradiction. Hence, we must have some $p\in \Gamma_c$ with $\partial_y\psi(p) < 0$.
\end{proof}

We are now in position to show that solutions to the overdetermined elliptic equation with one free boundary \eqref{eq:ellionefree} are one-dimensional.

\begin{proposition}\label{prop:rigidonefree}
Let $\psi $ a solution to \eqref{eq:ellionefree} in $D=\mathbb{T}\times(0,1)\setminus \overline{\Omega_{\Gamma_1}}$ such that all the streamlines of $\psi$ in $D$ are not singular. Then, $\psi = \psi(y)$.
\end{proposition}

\begin{proof}
We divide the proof in several steps.
\subsection*{Step 1} The first step in the proof is to extend $\psi\equiv 1$ in $D_0:= \mathbb{T}\times(0,1)\setminus \overline{D}$ and define
\begin{equation}
f_0(t) = -f(t) + g(t), \quad g(t) = \begin{cases}
f(1), & t = 1, \\
0, & t < 1.
\end{cases}
\end{equation}
It is clear to see that $\psi$ now weakly satisfies
\begin{equation}\label{eq:extendedellionefree}
\begin{cases}
-\Delta\psi = f_0(\psi), & \text{in }\mathbb{T}\times(0,1), \\
\psi(x,1) = 0, & x\in\mathbb{T}, \\
\psi(x,0) = 1, & x\in\mathbb{T},
\end{cases}
\end{equation}
and $\lbrace 0 < \psi <1 \rbrace = D$ is an open set. Note that from Lemma \ref{lemma:signf} we have $f(1)\leq 0$. If $f(1)=0$, then $g(t)=0$ and $\psi$ solves \eqref{eq:extendedellionefree} continuously. On the other hand, if $f(1)<0$, then, by Lemma \ref{lemma:streamlineregularity}, the free boundary $\Gamma_1$ is $C^2$ regular, we can integrate by parts and use the overdetermined boundary condition $\nabla\psi=0 $ on $\Gamma_1$ to show that $\psi$ solves \eqref{eq:extendedellionefree} in a weak sense. 

Now, let $\lambda\in \R$, and define the $\lambda$-translated reflection of $\psi$ on $\mathbb{T}\times[-1,1]$ by
\begin{equation}
\varphi_\lambda(x,y) = \begin{cases}
\psi(x,y), & y\geq 0 \\
\psi(x-\lambda,-y), & y< 0. 
\end{cases}
\end{equation}
It is then easy to see that $\varphi_\lambda$ satisfies
\begin{equation}
\begin{cases}
-\Delta\varphi_\lambda = f_0(\varphi_\lambda) & \text{ in }\mathbb{T}\times(-1,1) \\
\varphi_\lambda(x,-1) = \varphi_\lambda(x,1) =0, & x\in \mathbb{T}
\end{cases}
\end{equation}
and $\lbrace 0 < \psi < 1 \rbrace =D \cup D_{\lambda}$, where $D_\lambda\subset \mathbb{T}\times (-1,0)$ denotes the $\lambda$-translated reflection of $D$ with respect to the hyperplane $\lbrace y = 0 \rbrace$.

\begin{figure}[h!]
\centering
\includegraphics[width=0.8\textwidth]{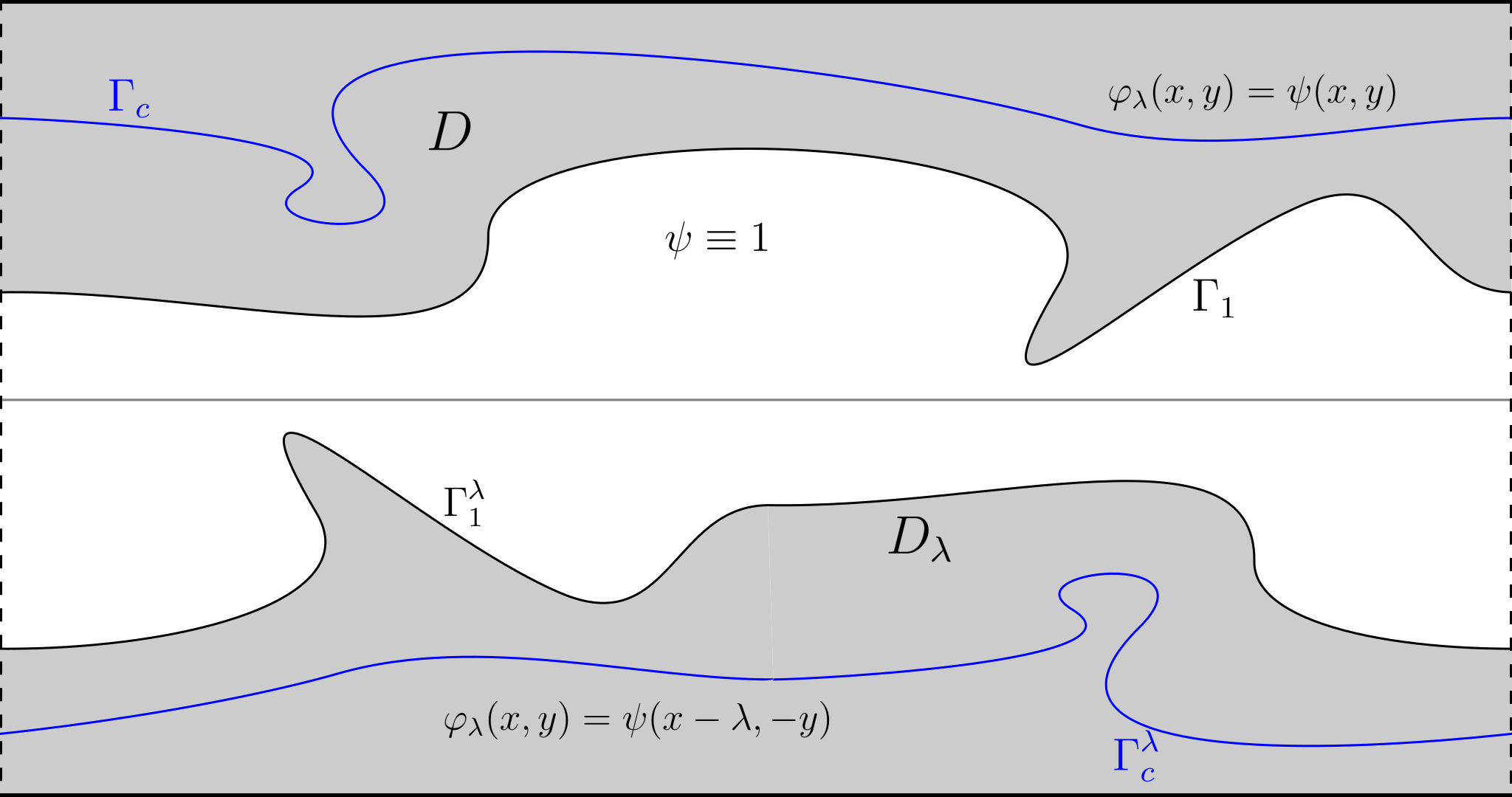}
\caption{The domain $D$ and $\psi$ are first reflected with respect to $y=0$ and then slided by a $\lambda$-translation in the $x$ direction. The free boundary $\Gamma_1$ and a regular streamline $\textcolor{blue}{\Gamma_c}$ are also reflected and slided into $\Gamma_1^\lambda$ and $\textcolor{blue}{\Gamma_c^\lambda}$, respectively. }
\end{figure}

\subsection*{Step 2} Using Theorem 7.2 of \cite{Brock}, note that $-f(t)$ is continuous and $g(t)$ is non-increasing, we see that $\varphi_\lambda$ is locally symmetric in the $y$ variable. Now, for all regular values $c\in(0,1)$, let $(x_0,y_0)\in \mathbb{T}\times(-1,0)$ such that $\varphi_\lambda(x_0,y_0) = c$ and $y_0$ is the largest $y<0$ for which $\varphi_\lambda(x,y) = c$, for some $x\in\mathbb{T}$. Let $\gamma_c(s)=(\gamma_c^1(s), \gamma_c^2(s))$ be a regular closed curve parametrizing the streamline $\lbrace \varphi_\lambda = c\rbrace$ such that $\gamma_c(0) = (x_0,y_0)$. Since $y_0$ is the maximal height of $\gamma_c$, there holds $\dot{\gamma}_c^2(0) = 0$. On the other hand, since $\varphi_\lambda(\gamma_c(s)) =c$ for all $s\in\mathbb{R}$, we deduce that
\begin{equation}
0 = \dot{\gamma}_c^1(s) \partial_x\varphi_\lambda(\gamma_c(s)) + \dot{\gamma}_c^2(s) \partial_y\varphi_\lambda(\gamma_c(s)) 
\end{equation}
so that for $s=0$, we must have $0=\partial_x\varphi_\lambda(\gamma_c(0)) = \partial_x\varphi_\lambda(x_0,y_0)$. Since $\nabla\psi\neq 0$ in regular streamlines, we conclude that $\partial_y\varphi_\lambda(x_0,y_0)\neq 0$, and thus $\partial_y\varphi_\lambda(x_0,y_0)>0$. Otherwise, we would have two distinct non-contractible stream-lines carrying the same value $0<c'<c$, both included in $\lbrace 0 < \varphi_\lambda < 1 \rbrace \cap \mathbb{T}\times(-1,0)$, so that there would exist a singular streamline, which is forbidden in $D$ and thus also in $D_\lambda$.

Now, the point $(x_0,y_0)$ and a small neighbourhood of it belong to a maximally connected component $U$ of $\lbrace 0 < \varphi_\lambda < 1\rbrace \cap \lbrace \partial_y\varphi_\lambda > 0 \rbrace$. By the local symmetry in $y$, the open set $U$ gets reflected with respect to some hyperplane $\lbrace y = y_\lambda \rbrace$. In particular, for $(x_0,y_0)$, there holds,
\begin{equation}
\varphi_\lambda(x_0,y_0) = \varphi_\lambda(x_0, 2y_\lambda - y_0), \quad	\partial_y \varphi_\lambda(x_0,y_0) = - \partial_y \varphi_\lambda(x_0,2y_\lambda - y_0)>0,
\end{equation}
with $2y_\lambda - y_0 > 0$.
Indeed, since $\varphi_\lambda(x_0,y_0) =c$ and $(x_0,y_0)$ denotes the closest point of $\Gamma_c$ in $D_\lambda$ to $\lbrace y = 0 \rbrace$, the unique $y>y_0$ satisfying $\varphi_\lambda(x_0,y_0)=\varphi_\lambda(x_0,y) < \varphi_\lambda(x_0,z)$, for all $z\in (y_0,y)$ must be in $D$, namely $y>0$. Moreover, $\partial_y \varphi_\lambda(x_0,y)<0$. Varying $\lambda\in \R$, we are sliding the streamline $\lbrace \varphi_\lambda = c \rbrace \cap D_\lambda$ together with its maximum $(x_0,y_0)$, so that $x_0$ ranges along all of $\mathbb{T}$. Using the local symmetry for each $\lambda\in \R$, we conclude that $\partial_y\varphi_\lambda(q)<0$ , for all $q\in\lbrace \varphi_\lambda = c \rbrace \cap D$, that is, $\partial_y\psi(q)<0$ , for all $q\in\lbrace \psi = c \rbrace$. Hence, by definition of $\varphi_\lambda$, we likewise obtain $\partial_y\varphi_\lambda > 0$ on $\lbrace \varphi_\lambda = c \rbrace \cap D_\lambda$. 

Let us explain this last argument more carefully. For all $x\in\mathbb{T}$ we define
\begin{equation}
y_x = \min \lbrace y\in (0,1) : (x,y) \in \Gamma_c \rbrace.
\end{equation}
and we claim that the map $Y(x)=y_x:\mathbb{T}\rightarrow (0,1)$ is a well-defined continuous mapping from $\mathbb{T}$ to $(0,1)$, so that $\gamma(x)=(x,y_x):\mathbb{T}\rightarrow\mathbb{T}\times(0,1)$ defines a continuous non-contractible closed curve such that $\psi(\gamma)=c$. We define $\Gamma = \lbrace \gamma(x) : x\in \mathbb{T} \rbrace$ as its image. Given the local symmetry of the solution in the vertical variable and the definition of $y_x$, we see that for all $x\in\mathbb{T}$ there exists some $\lambda	\in \R$ such that $\partial_y\psi(x,y_x) = - \partial_y\varphi_\lambda(x_0,y_0) < 0$ and  $\psi(x, y) > \psi(x,y_{x})=c$, for all $0<y < y_{x}$. Suppose now that the map $Y(x)$ is not continuous. Then, there exists a sequence $x_n\rightarrow x$ such that $y_{x_n} \rightarrow y_*\neq y_x$, up to a subsequence. Due to the local symmetry of solutions, and by the definition of each $y_{x_n}$, there holds
\begin{equation}
\psi(x_n, y) > \psi(x_n,y_{x_n})=c, \quad \text{for all } y\in (0, y_{x_n}).
\end{equation}
Now, by continuity of $\psi$, there holds $\psi(x, y) \geq \psi(x,y_*)=c$, for all $y < y_*$. By definition of $y_x$, we must have $y_*>y_x$, since $(x,y_*)\in \Gamma_c$. However, since $\partial_y\psi(x,y_x) < 0$, we can find some $y\in (y_x,y_*)$ such that $\psi(x,y)<\psi(x,y_x)=\psi(x,y_*)$, a contradiction. Hence, the map $Y(x)$ is continuous and $\Gamma$ is connected and non-contractible.

We now claim that $\Gamma_c = \Gamma$. Indeed, we shall see that $\Gamma$ is both relatively closed and open in $\Gamma_c$. Since $\Gamma_c$ is maximally connected, we then conclude that $\Gamma$ coincides with $\Gamma_c$. Since $Y(x)$ and $\psi$ are continuous, it is straightforward to see that $\Gamma$ is relatively closed in $\Gamma_c$. To see that $\Gamma$ is relatively open in $\Gamma_c$, for all $(x,y_x)\in \Gamma$ we shall find an open set $B$ such that $(x,y_x)\in B\cap \Gamma_c \subseteq \Gamma$. Let then $(\tilde{x},y_{\tilde{x}})\in \Gamma$ and note that, since $\partial_y\psi(\tilde{x},y_{\tilde{x}}) < 0$, by continuity there holds $\partial_y\psi < 0$ in $B_r(\tilde{x},y_{\tilde{x}})$, for some $r>0$. Choose now $\epsilon < r/2$ and, thanks to the continuity of $Y(x)$ at $\tilde{x}$, let $\delta_\epsilon>0$ be such that $|y_x - y_{\tilde{x}}|< \epsilon$ for all $|x-\tilde{x}|<\delta_\epsilon$. Further, let $\delta = \min \lbrace \delta_\epsilon, \epsilon\rbrace$ and define 
\begin{equation}
B = (\tilde{x} - \delta, \tilde{x} + \delta)\times (y_{\tilde{x}} - \epsilon, y_{\tilde{x}} + \epsilon).
\end{equation}
Now, there holds $B\cap \Gamma_c = \lbrace \gamma(x): \tilde{x}-\delta < x < \tilde{x} + \delta\rbrace \subset \Gamma$. Indeed, the inclusion $\supseteq$ is obvious due to the continuity of $Y(x)$. On the other hand, to show the inclusion $\subseteq$ assume towards a contradiction that there exists some $(x,y)\in B\cap \Gamma_c$ such that $y\neq Y(x)$. By definition of $Y(x)$, it must be $y>y_x=Y(x)$. Nevertheless, $\psi(x,y) = \psi(x,y_x)$, and thus there must exists some $y'\in (y_x,y)$ such that $\partial_y\psi(x,y')=0$. However, this contradicts $\partial_y\psi< 0$ in $B\subset B_r(\tilde{x}, y_{\tilde{x}})$. Hence, $\Gamma\subseteq \Gamma_c$ is both relatively open and closed in the connected  set $\Gamma_c$. We then conclude that $\Gamma_c = \Gamma$ and thus $\partial_y\psi < 0$ in $\Gamma$, since each $(x,y_x)$ is such that $\partial_y\psi(x,y_x) = -\partial_y\varphi_\lambda(x_0,y_0) < 0$. In fact, once we show that $Y(x)$ is continuous, we can use the $C^2$ regularity of $\psi$ to show that actually $Y(x)$, and thus $\gamma(x)$, is of class $C^1$, with $Y'(x) = - \frac{\partial_x\psi(x,y_x)}{\partial_y\psi(x,y_x)} = \frac{\partial_x\psi(x,y_x)}{\partial_y\varphi_\lambda(x_0,y_0)}$. 

\subsection*{Step 3} Let us now argue for those $c\in (0,1)$ corresponding to regular-singular streamlines. By density of regular values we see that for all points $p\in D_\lambda$, there holds $\partial_y\varphi_\lambda(p) \geq 0 $. Therefore, we can apply Lemma \ref{lemma:technical} for any regular-singular streamline $\Gamma_c$ in $D_\lambda$ and we conclude the existence of some point $q$ in $\Gamma_c\cap D_\lambda$ such that $\partial_y \varphi_\lambda(q)>0$. Again by density of regular values, there exists an open neighbourhood $B_\delta(q)$ of $q$ for which $\partial_y\varphi_\lambda >0$ and $\varphi_\lambda$ attains regular values, slightly smaller than $c$ for regular streamlines lying below $\Gamma_c$ and slightly larger than $c$ for regular streamlines lying above $\Gamma_c$. Each of these regular values lie, a priori, on two different maximally connected components of $\lbrace 0 < \varphi_\lambda < 1 \rbrace \cap \lbrace \partial_y \varphi_\lambda > 0 \rbrace$, and they both intersect the open ball $B_\delta(q)$, where $\partial_y\varphi_\lambda > 0$. Hence, these regular streamlines lying slightly above and slightly below $\Gamma_c$, together with the open ball $B_\delta(q)$ must belong to the same maximally connected component of $\lbrace 0 < \varphi_\lambda < 1 \rbrace \cap \lbrace \partial_y \varphi_\lambda > 0 \rbrace$, and thus they are reflected to $D$ with respect to the same hyperplane of reflection. As such, $q\in \Gamma_c\cap D_\lambda$ is reflected to $D$ and meets $q'\in \Gamma_c\cap D$, with $\partial_y\psi(q') = \partial_y\varphi_\lambda(q') = - \partial_y\varphi_\lambda(q) < 0$. As before, we slide the point $q$ over all $\Gamma_c\cap D_\lambda$ by varying $\lambda\in \R$, obtaining by local symmetry a closed non-contractible curve $\Gamma\subseteq\Gamma_c\cap D$. Since $\partial_y\psi\leq 0$ in $D$ and $\partial_y\psi < 0$ on $\Gamma$, we conclude that $\Gamma = \Gamma_c\cap D$ and thus $\partial_y\psi <0$ in any regular-singular streamline $\Gamma_c\cap D$. In other words, there can be no regular-singular streamlines, and all streamlines are regular. This shows that the maximal connected component of $\lbrace 0 < \varphi_\lambda < 1 \rbrace \cap \lbrace \partial_y \varphi_\lambda > 0 \rbrace$ is precisely $D_\lambda$, namely the $\lambda$-translated reflection of $D$ with respect to the hyperplane $\lbrace y = 0 \rbrace$. 

\subsection*{Step 4} Since $\partial_y\psi < 0$ in $D$, all stream-lines are given by graphs in the $x$-variable, monotonically foliating $D$ and its reflection. Moreover, $\varphi_\lambda$ is symmetric with respect to the hyperplane $\lbrace y = y_\lambda \rbrace$. We shall see that $y_\lambda = 0$, for all $\lambda\in \R$. To see this, let $y_m>0$ denote the minimum value of the vertical component of the free boundary $\Gamma_1$, and assume that the hyperplane of symmetry is given by $\lbrace y = y_\lambda\rbrace$, for some $y_\lambda <0$. Firstly, $y_\lambda\geq-y_m$, since otherwise we would have $\partial_y\varphi_\lambda(x_0,y_\lambda) = 0$ for some interior point $(x_0,y_\lambda)\in D_\lambda$. Let $\delta>0$ such that $\delta< 1-y_m$ and $\delta< -2y_\lambda$. Then, let $(x_0,-y_m)\in (\Gamma_1)_\lambda$, which denotes the $\lambda$-translated reflection of the streamline $\Gamma_1$ with respect to $\lbrace y = 0 \rbrace$ and observe that
\begin{equation}
1 > c = \varphi_\lambda(x_0,-y_m-\delta) = \varphi_\lambda (x_0,2y_\lambda - (-y_m-\delta)) = \varphi_\lambda(x_0, y_m + 2y_\lambda + \delta) = 1,
\end{equation}
because $y_m+2y_\lambda + \delta < y_m$ and thus $(x_0,y_m+2y_\lambda + \delta)$ falls within the domain where $\psi$ is extended to be 1, thus reaching a contradiction. For $y_\lambda\in (0,y_m)$ we argue similarly, taking $0<\delta < 2y_\lambda$ and $(x_0,y_m)\in \Gamma_1$. Hence,
\begin{equation}
1 > \varphi_\lambda(x_0, y_m + \delta) = \varphi_\lambda (x_0, 2y_\lambda - y_m - \delta) = 1,
\end{equation} 
because $-y_m + 2y_\lambda + \delta > -y_m$  and thus $(x_0, 2y_\lambda - y_m - \delta)$ falls within the region where $\psi$ is equal to 1. Therefore, $y_\lambda = 0$, for all $\lambda\in \R$, which shows that $\varphi_\lambda$ is symmetric with respect to the hyperplane $\lbrace y = 0 \rbrace$. In other words, for all $(x,y)\in D$, there holds
\begin{equation}
\psi(x,y) = \varphi_\lambda(x,y) = \varphi_\lambda(x,-y) = \psi(x-\lambda, y),
\end{equation}
for all $\lambda\in\mathbb{R}$. Hence, we conclude that $\psi = \psi(y)$ alone.
\end{proof}

We finish the section showing that solutions to the overdetermined boundary value problem \eqref{eq:ellitwofree} are one-dimensional, as well.

\begin{proposition}\label{prop:rigidtwofree}
Let $\psi$ be a solution of \eqref{eq:ellitwofree} in $D=\Omega_{\Gamma_0}\setminus\overline{\Omega_{\Gamma_1}}$ such that all streamlines of $\psi$ in $D$ are not singular. Then, $\psi = \psi(y)$.
\end{proposition}
\begin{proof}
The first step in the proof is to extend $\psi$ from $D$ to $\mathbb{T}\times(0,1)$. We do so by setting $\psi\equiv 0$ on $\mathbb{T}\times(0,1)\setminus \overline{\Omega_{\Gamma_0}}$ and setting $\psi\equiv 1$ on $\mathbb{T}\times(0,1) \setminus \overline{\mathbb{T}\times(0,1)\setminus \overline{\Omega_{\Gamma_1}}}$. Next, we define
\begin{equation}
f_0(t) = -f(t) + g(t), \quad g(t) = \begin{cases} f(1), & t=1, \\ 0, & 0<t< 1, \\
f(0), & t = 0 \end{cases},
\end{equation}
and since $f(1)\leq 0 \leq f(0)$, we observe that $g(t)$ is non-increasing. Moreover, $\psi$ is a weak solution of 
\begin{equation}
\begin{cases}
-\Delta\psi = f_0(\psi), & \text{in }\mathbb{T}\times(0,1), \\
\psi(x,1) = 0, & x\in\mathbb{T}, \\
\psi(x,0) = 1, & x\in\mathbb{T},
\end{cases}
\end{equation}
while $D=\lbrace 0 < \psi <1 \rbrace$ is an open set. As before, Lemma \ref{lemma:streamlineregularity} ensures $C^2$ regularity of $\Gamma_i$ in case $f(i)\neq 0$. Next, for all $\lambda\in \R$ we define the $\lambda$-translated reflection of $\psi$ on $\mathbb{T}\times[-1,1]$ by
\begin{equation}
\varphi_\lambda(x,y) = \begin{cases}
\psi(x,y), & y\geq 0 \\
\psi(x-\lambda,-y), & y< 0
\end{cases}
\end{equation}
so that $\varphi_\lambda$ satisfies
\begin{equation}
\begin{cases}
-\Delta\varphi_\lambda = f_0(\varphi_\lambda) & \text{ in }\mathbb{T}\times(-1,1) \\
\varphi_\lambda(x,-1) = \varphi_\lambda(x,1) =0, & x\in \mathbb{T}
\end{cases}
\end{equation}
and $\lbrace 0 < \psi < 1 \rbrace =D \cup D_{\lambda}$, where $D_\lambda\subset \mathbb{T}\times (-1,0)$ denotes the $\lambda$-translated reflection of $D$ with respect to the hyperplane $\lbrace y = 0 \rbrace$. From here, the proof of the proposition is identical to the proof of Proposition \ref{prop:rigidonefree}, we omit the details.
\end{proof}

\section{Over-determined elliptic equations in the annulus}\label{sec:rigidoverdetannulus}
In this section we assume $M=B_R\setminus \overline{B_r}$ is an annulus, for some $0 < r < R$. There is a rich literature on solutions satisfying over-determined elliptic boundary problems in annular domains, and so we shall draw from them, with the necessary modifications to our purposes. To deal with the equations in the annulus we first recall, see \cite{Brock}, the concept of local radial symmetry.
\begin{defi}
Assume $\psi\geq 0$ admits the decomposition
\begin{equation}
\lbrace 0 < \psi < \sup \psi \rbrace = \bigcup_{k\geq 1}U_k \bigcup S,
\end{equation} 
where $\nabla\psi = 0$ in $S$ and the $U_k$'s are pairwise disjoint annuli $B_{R_k}(\x_k)\setminus \overline{B_{r_k}(\x_k)}$, with $0\leq r_k < R_k$ and $\x_k \in \lbrace 0 < \psi < \sup \psi \rbrace$. Assume further that $\psi$ is radially symmetric in $U_k$, namely
\begin{equation}
\psi(\x) = \psi(|\x-\x_k|), \quad \frac{\partial\psi}{\partial\rho}(\x) < 0,
\end{equation}
for all $\x\in U_k$, where $\rho=|\x-\x_k|$, and $\psi(\x) \geq \psi|_{\partial B_{r_k}(\x_k)}$, for all $k\geq 1$. Then, we say that \emph{$\psi$ is locally radially symmetric}. 
\end{defi}

In this section we show that solutions to \eqref{eq:elliM}, \eqref{eq:ellionefree} and \eqref{eq:ellitwofree} are locally radially symmetric. Exploiting the structure of the streamlines in the fluid sub-domains, we are able to prove that the set $S$ is empty and the solutions are \textit{globally} radially symmetric. We begin by showing that solutions to the semilinear elliptic equation \eqref{eq:elliM} for the annular domain $M=B_R\setminus \overline{B_r}$ must be radial. Now, the problem does not, a priori, have over-determined boundary conditions, but the boundaries are already circular sets. The following result exploits this fact to conclude solutions are radial.

\begin{proposition}\label{prop:rigidityannulus}
Let $M=B_R\setminus \overline{B_r}$ and $\psi$ be a solution of \eqref{eq:elliM}. Then, $\psi=\psi(r)$.
\end{proposition}

\begin{proof}
We extend $\psi \equiv 1$ in $B_r$ so that now we have that $\psi\in H^1_0(B_R)$ satisfies
\begin{equation}
\begin{cases}
\Delta\psi = f(\psi), \quad  0<\psi<1, & \text{in }M=B_R\setminus B_{r}, \\
\psi \equiv 1, & \text{in }\overline{B_r},
\end{cases}
\end{equation}
for some $f\in C[0,1]$, with $M=\lbrace 0 < \psi < 1 \rbrace$ and $\psi\in C^3(M)$. Hence, we can apply Corollary 7.3 of \cite{Brock} to conclude that $\psi$ is locally symmetric. Hence, 
\begin{equation}
M = \bigcup_{k\geq 1}U_k \bigcup S,
\end{equation} 
where $\nabla\psi = 0$ in $S$ and the $U_k$'s are pairwise disjoint annuli $B_{r_k}(\x_k)\setminus \overline{B_{r_k}(\x_k)}$, with $0\leq r_k < R_k$ and $\x_k \in D$. Moreover, $\psi$ is radially symmetric in $U_k$. Any regular streamline must be entirely contained in only one of these $U_k$'s, since regular streamlines are connected sets with $\nabla\psi\neq 0$ on them and the $U_k$'s are pairwise disjoint. Moreover, since all streamlines are closed and non-contractible loops, these annuli $U_k$ must all be nested and their centres $\x_k$ must lie within $B_r$.

On the other hand, the set $S$ may, a priori, be non-empty, given the possible existence of regular-singular streamlines. However, we shall see this is not the case. Indeed, assume $\Gamma_c$ is a regular-singular streamline and let $q\in\Gamma_c$ be such that $\nabla\psi(q)\neq 0$. Then, $q\in U_k$ for some $k\geq 1$, so that $r_k < |q-\x_k|<R_k$. Therefore, the set $\Gamma := \lbrace \x\in M: |x-\x_k| = |q-\x_k|\rbrace$ is a closed non-contractible circle, contained in $U_k$ and, more importantly, it is such that $\psi(\x) = \psi(q) =c$, for all $\x\in \Gamma$. Therefore, it must be $\Gamma_c=\Gamma$ and $\partial_\rho\psi < 0$ in $\Gamma_c$, thus concluding that there are no regular-singular streamlines. Hence, $S$ is empty and since $M$ is connected, the union is made of only one annuli, $M$ itself. Therefore, $\psi$ is radially symmetric and decreasing in $M$, $u$ is a radial flow.
\end{proof}

We next show that solutions to the semilinear elliptic problem \eqref{eq:ellionefree} with one overdetermined boundary must be circular.

\begin{proposition}\label{prop:rigidannulusonefree}
Let $\psi$ a solution to \eqref{eq:ellionefree} in $D=M\setminus \overline{\Omega_{\Gamma_1}}$ such that all the streamlines of $\psi$ in $D$ are not singular. Then, $\psi = \psi(r)$.
\end{proposition}

\begin{proof}
Firstly, we extend $\psi\equiv 1$ in $\Omega_{\Gamma_1}$ and define
\begin{equation}
f_0(t) = -f(t) + g(t), \quad g(t) = \begin{cases}
f(1), & t = 1, \\
0, & t < 1.
\end{cases}
\end{equation}
It is clear to see that $\psi$ now weakly satisfies
\begin{equation}\label{eq:annulusextendedellionefree}
\begin{cases}
-\Delta\psi = f_0(\psi), & \text{in }B_R, \\
\psi = 0, & \text{on }\partial B_R, \\
\psi\geq  0, & \text{in }B_R
\end{cases}
\end{equation}
and $\lbrace 0 < \psi <1 \rbrace = D$ is an open set, with $\psi\in C^1(D)$. Note that since $f(t)$ is continuous and $g(t)$ is non-increasing, we can use Theorem 7.2 of \cite{Brock} to deduce that $\psi$ is locally symmetric. As in the proof of Proposition \ref{prop:rigidityannulus}, we show, thanks to the local radial symmetry, that any regular-singular streamline is in fact regular. Hence the set $S$ is empty and $\psi$ is globally radially symmetric and decreasing, we omit the details.
\end{proof}

We finish the section proving that solutions to the overdetermined boundary value problem \eqref{eq:ellitwofree} must also be one-dimensional and confined in an annular domains.
\begin{proposition}\label{prop:rigidannulustwofree}
Let $\psi$ a solution to \eqref{eq:ellitwofree} in $D=\Omega_{\Gamma_0}\setminus \overline{\Omega_{\Gamma_1}}$ such that all the streamlines of $\psi$ in $D$ are not singular. Then, $D$ is an annulus and $\psi$ is radially symmetric and decreasing in $D$.
\end{proposition}

\begin{proof}
The strategy of the proof is similar to that of Proposition \ref{prop:rigidannulusonefree}. We first extend $\psi\equiv 1$ in $\Omega_{\Gamma_1}$ and $\psi \equiv 0$ in $M\setminus \overline{D}$. Define as well
\begin{equation}
f_0(t) = -f(t) + g(t), \quad g(t) = \begin{cases}
f(1), & t = 1, \\
0, &  0 < t < 1 \\
f(0), & t = 0.
\end{cases}
\end{equation}
so that $\psi$ now weakly satisfies
\begin{equation}
\begin{cases}
-\Delta\psi = f_0(\psi), & \text{in }B_R, \\
\psi = 0, & \text{on }\partial B_R, \\
\psi\geq  0, & \text{in }B_R
\end{cases}
\end{equation}
and $\lbrace 0 < \psi <1 \rbrace = D$ is an open set, with $\psi\in C^1(D)$. Since $f(t)$ is continuous and $g(t)$ is non-increasing, we can again apply Theorem 7.2 of \cite{Brock} to deduce that $\psi$ is locally symmetric. The arguments to conclude that $S$ is empty, $D$ is an annular domain and $\psi$ is radially symmetric and decreasing there are the same to the ones in the proof of Proposition \ref{prop:rigidityannulus}, we omit the details.
\end{proof}

\section{Free Boundary Problems: proof of Corollary \ref{cor:rigidfreeboundary}}\label{sec:rigidfreeboundary}
In this section we generalize Theorem \ref{thm:rigidity} to more general annular domains in either $\R^2$ or $\mathbb{T}\times \R$ in Corollary \ref{cor:rigidfreeboundary}. In order to conclude that stationary solutions must either be radial or shear, since now the boundaries are no longer assumed to be circular, respectively horizontal, and keeping in mind the flexibility results of \cite{CDG21}, we must assume enhanced boundary conditions of the velocity field. 

We consider here the free-boundary Euler equations posed in an open bounded domain $M=\Omega_{\Gamma_0}\setminus \overline{\Omega_{\Gamma_1}}$, where $\Gamma_0$ and $\Gamma_1$ denote two non intersecting closed non-contractible curves either in $\mathbb{T}\times \R$ or in $\R^2$. The equations read
\begin{align}\label{eq:freeboundaryEuler}
u\cdot \nabla u &= -\nabla p\quad\  \text{in} \ M,\\
\nabla \cdot u &=0 \quad\quad \ \ \   \text{in} \ M,\\
 u\cdot \hat{n} &=0 \quad\quad \ \ \   \text{on} \ \partial M=\Gamma_0 \cup \Gamma_1, \\
|u| &=a_i \quad\quad \ \ \   \text{on} \ \Gamma_i,
\end{align}
for some $a_i\geq 0$. Given the general strategy for the proof of Theorem \ref{thm:rigidity}, the proof of Corollary \ref{cor:rigidfreeboundary} is now reduced to establishing rigidity results for the stream-function $\psi$, which now solves a modified version of \eqref{eq:ellitwofree}, namely
\begin{equation}\label{eq:ellitwofreectn}
\begin{cases}
\Delta\psi = f(\psi), \quad  0<\psi<1, & \text{ in }D=\Omega_{\Gamma_0} \setminus \overline{\Omega_{\Gamma_1}}, \\
\psi = 0, \quad |\nabla\psi| = a_0, & \text{ on }\Gamma_0, \\
\psi = 1, \quad |\nabla\psi| = a_1, & \text{ on }\Gamma_1,
\end{cases}
\end{equation}
for some $f\in C[0,1]\cap C^1(0,1)$, with $a_0,a_1\geq 0$ and the further assumption that $D=\Omega_{\Gamma_0} \setminus \overline{\Omega_{\Gamma_1}}$ contains no singular streamlines of $\psi$. In case $M$ is an annular domain in $\R^2$, Corollary \ref{cor:rigidfreeboundary} is a direct consequence of Theorem 1.15 in \cite{WangZhan}. A close inspection of the proof shows that the assumption $\nabla\psi\neq 0$ in $D$ there can be replaced by the assumption of $\psi$ being laminar and $D$ containing no singular stream-lines of $\psi$. As we have previously remarked, see Lemma \ref{lemma:mainellipticeq}, this already ensures the existence of a $C^1(0,1)$ function $f$ such that \eqref{eq:ellitwofreectn} is satisfied. We now focus our discussion for subsets $M$ of the periodic strip.

\subsection{Degenerate free boundary problems in the periodic channel}
We assume $M\subset \mathbb{T}\times \R$ is diffeomorphic to a periodic channel. If $a_0=a_1=0$, we can directly use Proposition \ref{prop:rigidtwofree} to conclude that $u$ is a shear flow. In what follows we assume that the $a_i$'s are not zero simultaneously. Consider first that $a_1=0$ and $a_0>0$. By rotational invariance of the Laplacian, the case $a_1>0$ and $a_0=0$ is treated rotating the domain by $\pi$ degrees. In this subsection we further assume the degenerate case $f(1)=0$, which corresponds to assuming that the vorticity is also zero on the free boundary. We have the following result.
\begin{proposition}
Let $\psi$ be a solution to  
\begin{equation}
\begin{cases}
\Delta\psi = f(\psi), \quad  0<\psi<1, & \text{ in }D=\Omega_{\Gamma_0} \setminus \overline{\Omega_{\Gamma_1}}, \\
\psi = 0, \quad |\nabla\psi| = a_0>0, & \text{ on }\Gamma_0, \\
\psi = 1, \quad |\nabla\psi| = 0, & \text{ on }\Gamma_1,
\end{cases}
\end{equation}
for some $f\in C([0,1])\cap C^1(0,1)$ and assume that $D$ contains no singular streamlines of $\psi$. Then, $\psi = \psi(y)$. 
\end{proposition}

\begin{proof}
The proof is based on a local symmetry result of Brock in \cite{Brock2}, see also Theorem 1.15 in \cite{WangZhan}. Assume without loss of generality that $D$ embeds in the periodic channel, namely that $D\subset M=\mathbb{T}\times[0,1]$, and denote $M_0 = \mathbb{T}\times\lbrace 0 \rbrace$ the bottom boundary of $M$.  We extend $\psi\equiv 1$ the region in $\mathbb{T}\times[0,1]$ bounded above by $\Gamma_1$ and we denote $D_0=\Omega_{\Gamma_0}\setminus \overline{\Omega_{M_0}}$, the region in $M$ bounded above by $\Gamma_0$. Since $\nabla\psi=0$ on $\Gamma_1$ and $f(1) = 0$, we see that $\psi$ now satisfies
\begin{equation}
\begin{cases}
\Delta\psi = f(\psi), \quad  0<\psi<1, & \text{ in }D_0 , \\
\psi = 0, \quad |\nabla\psi| = a_0>0, & \text{ on }\Gamma_0, \\
\end{cases}
\end{equation}
pointwise, with $f\in C([0,1])$. As in the proof of Proposition \ref{prop:rigidonefree}, let $\lambda\in \R$ and denote $\varphi_\lambda$ the $\lambda$-translated reflection of $\psi$,  that is
\begin{equation}
\varphi_\lambda(x,y) = \begin{cases}
\psi(x,y), & (x,y)\in D_0 \\
\psi(x-\lambda,-y), & (x,-y)\in D_0,
\end{cases}
\end{equation}
which is defined in $\mathcal{D} = D_0\cup \pi_\lambda(D_0)$, where $\pi_\lambda(D_0)$ denotes the $\lambda$-translated in $x$ symmetric reflection of $D$ with respect to $y=0$. Now, $\mathcal{D}$ is connected and $\varphi_\lambda\in C^2_0(\mathcal{D})$ is a strong solution to
\begin{equation}
\begin{cases}
\Delta\varphi_\lambda = f(\varphi_\lambda) & \text{ in }\mathcal{D} \\
\varphi_\lambda >0  & \text{ in }\mathcal{D} \\
\varphi_\lambda = 0, &\text{ on } \partial \mathcal{D}.
\end{cases}
\end{equation}
Moreover, $\lbrace 0 < \varphi_\lambda < 1 \rbrace =\mathcal{D}$. The proof of Theorem 1 in \cite{Brock2}, see also the proof of Proposition 5.3 in \cite{WangZhan}, shows that for all $\epsilon>0$ we have that $(\varphi_\lambda - \epsilon)_+$ is locally symmetric in the $y$ variable, for all $\lambda\in \R$. Hence, arguing as in the proof of Proposition \ref{prop:rigidonefree}, we conclude that $(\varphi_\lambda-\epsilon)_+$ is symmetric with respect to the horizontal axis $\lbrace y = 0 \rbrace$, for all $\lambda\in \R$. Therefore, for  all $(x,y)\in D_\varepsilon = \lbrace \epsilon < \psi < 1 \rbrace $, there holds
\begin{equation}
\psi(x,y) = \varphi_\lambda(x,y) = \varphi_\lambda(x,-y) = \psi(x-\lambda, y),
\end{equation}
for all $\lambda\in\mathbb{R}$. Hence, we conclude that $\psi = \psi(y)$ alone in $D_\epsilon$. Moreover, since $0<\psi < 1$ in $D$, we see that $D_\epsilon$ exhaust $D$ and, as a result, we obtain $\psi = \psi(y)$ in $D$. The proof is finished.
\end{proof}
\subsection{A moving plane argument}
To finish the proof of Corollary \ref{cor:rigidfreeboundary}, we still have to consider two more cases: either $a_1=0$, $f(1)\neq 0$ and $a_0>0$ or $a_1,a_0>0$. These can be treated jointly. To conclude translational invariance of solutions to \eqref{eq:ellitwofreectn}, we will use the sliding method presented in the proof of Proposition \ref{prop:rigiditychannel}. To do that, we need to adapt the sliding method to the boundaries not being horizontal and we first need to show that the boundaries are given by $C^2$ graphs in the periodic variable $x\in\mathbb{T}$. This is the purpose of the following result, which uses a moving plane argument.
\begin{lemma}\label{lemma:graphlikeboundary}
Let $\psi$ be a solution to  
\begin{equation}
\begin{cases}
\Delta\psi = f(\psi), \quad  0<\psi<1, & \text{ in }D=\Omega_{\Gamma_0} \setminus \overline{\Omega_{\Gamma_1}}, \\
\psi = 0, \quad |\nabla\psi| = a_0, & \text{ on }\Gamma_0, \\
\psi = 1, \quad |\nabla\psi| = a_1, & \text{ on }\Gamma_1,
\end{cases}
\end{equation}
for some $f\in C^1([0,1])$. Assume that $D$ contains no singular streamlines of $\psi$ and that $f(i)\neq 0$ if $a_i=0$. Then, $\Gamma_i = \lbrace (x,h_i(x)):\, x\in \mathbb{T} \rbrace$, for some $h_i\in C^2(\mathbb{T})$, for $i=0,1$.
\end{lemma}

\begin{proof}
We will show that $\Gamma_0$ is given by a graph function. It will be clear from the arguments that the same conclusion applies for $\Gamma_1$. In what follows, we implement a moving plane argument to show that $\partial_y\psi < 0$ on $\Gamma_0$, since this is enough to then conclude, via the Implicit Function Theorem, that $\Gamma_0 = \lbrace (x,h_0(x)):\, x\in\mathbb{T}\rbrace$, for some $h_0\in C^2(\mathbb{T})$. To that purpose, since both $\Gamma_i$ are bounded closed curves, we embed $D$ in a periodic channel $M$, which we assume to be $M=\mathbb{T}\times(0, 1)$. We then extend $\psi \equiv 1$ in $\Omega_{\Gamma_1}$, the region in $M$ bounded by $\Gamma_1$. For any $\lambda\geq 0$, we define
\begin{equation}
H_\lambda = \lbrace y = \lambda \rbrace, \quad H_\lambda^+ = \lbrace y > \lambda \rbrace, \quad H_\lambda^-  \lbrace y  < \lambda \rbrace,
\end{equation}
and, for any set $A\subset \mathbb{T}\times \R$, we denote
\begin{equation}
A_\lambda^+ = A\cap H_\lambda^+, \quad A_\lambda^- = A\cap H_\lambda^-.
\end{equation}
Next, let $\pi_\lambda$ denote the reflection with respect to the hyperplane $H_\lambda$ and, to ease notation, we denote $\Omega_i:=\Omega_{\Gamma_i}$. Now, let $\mu_m =\min\lbrace y:\, (x,y)\in \Gamma_0 \rbrace$ and define
\begin{equation}
I=\lbrace \mu\geq \mu_m :\, \forall \lambda\geq \mu, \, \pi_\lambda (\overline{\Omega_{i,\lambda}^+}) \subset \Omega_i, \quad (-1)^i\nu_2(p)> 0, \,\forall p\in \Gamma_i\cap H_\lambda,\, i=0,1 \rbrace.
\end{equation}
Here $\nu_2(p)$ denotes the second component of the exterior normal unit vector of $\Omega_0$. The set $I$ denotes all $\mu\in\R$ for which the reflection of $\Omega_{i,\lambda}^+$ fall strictly inside $\Omega_i$. Clearly, $I$ is bounded below and non-empty, since $\mu_1:=\sup\lbrace y: (x,y)\in D \rbrace$ belongs to it. Let $\mu_0:=\inf I$ and assume $\mu_0 > \mu_m$. Then, the infimum takes place when either the reflected set becomes tangential to its boundary, or when the boundary becomes parallel to the vertical vector $\mathbf{e}_2 = (0,1)$. As in the proof of Proposition \ref{prop:rigiditychannel}, for all $\lambda\in I$ we define the comparison function
\begin{equation}
w^\lambda(\x) = \psi(\pi_\lambda(\x)) - \psi(\x),
\end{equation}
for all $\x\in\Sigma_\lambda := \Omega_{0,\lambda}^+\setminus \pi_\lambda(\overline{\Omega_{1,\lambda}^-})$,  the set of all points in $\Omega_{0,\lambda}^+$ such that their reflection does not fall into $\overline{\Omega_{1,\lambda}^-}$. In particular, $\pi_\lambda(\x) \neq \overline{\Omega_1}$, for all $\x\in \Sigma(\lambda)$, since $\pi_{\lambda}(\Omega_{0,\lambda}^+)\cap \Omega_0 = \emptyset$, for all $\x\in \Sigma(\lambda)$. Moreover, we set 
\begin{equation}
J=\lbrace \lambda\in I: w^\lambda(\x) > 0, \quad\forall \x\in \Sigma(\lambda) \rbrace
\end{equation}
and we claim that $J$ is non-empty. Indeed, it can be seen  that $(\mu_1-\epsilon, \mu_1)\subset J$, for some $\epsilon>0$ sufficiently small. This is clear since $-\partial_y\psi =\partial_n\psi = |u|=a_0> 0$ at the maximal point on $\Gamma_0$, and thus by continuity $\partial_y\psi<0$ in $\Omega_{0,\lambda}^+$, for all $\lambda<\mu_1$ sufficiently close to $\mu_1$.

We next define $\lambda_0=\inf J$ and claim that $\lambda_0 \geq \mu_0$. If not, $\lambda_0 > \mu_0$ gives a contradiction as in the proof of Theorem 3 in \cite{Reichel}. One reaches a contradiction after using the maximum principle and the Hopf Lemma on balls in $D$, where $0<\psi<1$ and $f$ is known to be $C^1(0,1)$. We omit the details. We then conclude that $\lambda_0=\mu_0$. Since we assume $\mu_0 > \mu_m$, one of the following is true:
\subsection*{Case 1. Internal tangency} There is some $\x_0\in \Gamma_i^+$ such that $\pi_{\mu_0}(\x_0) \in \Gamma_i$, for $i=0$ or $i=1$. By the overdetermined boundary condition $|u|=a_i>0$ on $\Gamma_i$, we have
\begin{equation}
\partial_\nu w^{\mu_0}(\x_0) = 0.
\end{equation}
Due to the $C^2$ regularity of $\Gamma_i$, which is straightforward if $a_i\neq 0$, and follows from Lemma \ref{lemma:streamlineregularity} if $a_i=0$ since then $f(i)\neq 0$, we can take a sufficiently small ball $B_r(\x_*)$ in $D$ tangent to $\Gamma_i$ at $\x_0$, so that $B_r, \pi_{\mu_0}(B_r)\subset D$. In fact, since $\pi_{\mu_0}$ is an isometry,
\begin{equation}
d_D(\x) \geq r - |\x-\x_*|, \quad d_D(\pi_{\mu_0}(\x))\geq r - |\x-\x_*|,
\end{equation}
for all $\x\in B_r(\x_*)$. Here, $d_D(\x):=\text{dist}(x,\partial D)$. Since $f(i)\neq 0$ whenever $a_i=0$, we can apply Lemma 5.2 of \cite{Ruiz} to obtain
\begin{equation}\label{eq:quasiLipf}
|f(\psi(\x_1)) - f(\psi(\x_2))| \leq \frac{C}{\min\lbrace d_D(\x_1),\, d_D(\x_2)\rbrace }|\psi(\x_1) - \psi(\x_2)|, \quad \forall \x_1,\,\x_2\in D,
\end{equation}
for some $C>0$. Together with \eqref{eq:quasiLipf}, we conclude that 
\begin{equation}
\left| g(\x) \right| = \left| \frac{f(\psi(\pi_{\mu_0}(\x))) - f(\psi(\x))}{\psi(\pi_{\mu_0}(\x)) - \psi(\x)} \right| \leq \frac{C}{r-|\x-\x_*|},
\end{equation}
for all $\x\in B_r(\x_*)$. Then, since $d_{B_r}(\x) = r-|\x-\x_*|$, there holds $g(\x)d_{B_r}(\x)\in L^\infty(B_r(\x_*))$ and since $w^{\mu_0}\geq 0$ in $B_r(\x_*)$, by the Hopf Lemma for singular operators, see \cite[Proposition 4.3]{Ruiz}, we conclude that $w^{\mu_0}\equiv 0$ in $B$, and then by unique continuation in $D$ thus reaching a contradiction with $0 < \psi < 1$ in $D$, $\psi=0$ on $\Gamma_0$ and $\psi = 1$ on $\Gamma_1$.

\subsection*{Case 2. Orthogonality of $\Gamma_i$ and $H_{\mu_0}$} Now, there is some $\x_0\in \Gamma_i$ such that $\nu_2(\x_0) = 0$. Proceeding as in \cite{serrin}, see also \cite{Ruiz}, there holds $\nabla^2 w^{\mu_0}(\x_0)=0$ thanks to the over-determined boundary conditions $|\nabla\psi|=a_i$ in $\Gamma_i$. Hence, as before, we can place a sufficiently small ball $B_r$ in $D$ tangent to $D$ at $\x_0$ and such that $B, \pi_{\mu_0}(B)\subset D$. If $a_i\neq 0$, then $f$ is of class $C^1$ up to the boundary and by the Serrin Corner Lemma, see \cite{serrin}, we conclude that $w^{\mu_0}\equiv 0$ in $B_r$ and then in all $D$. This again gives a contradiction with $0 < \psi < 1$ in $D$, $\psi=0$ on $\Gamma_0$ and $\psi = 1$ on $\Gamma_1$. Instead, if $a_i=0$, then $f(i)\neq 0$ and we can likewise use Serrin Corner Lemma for singular operators, see \cite[Proposition 4.4]{Ruiz} to conclude that $w^{\mu_0}\equiv 0$ in $B_r$, thus leading to a contradiction as well.

\begin{figure}[h!]
\centering
\includegraphics[width=0.8\textwidth]{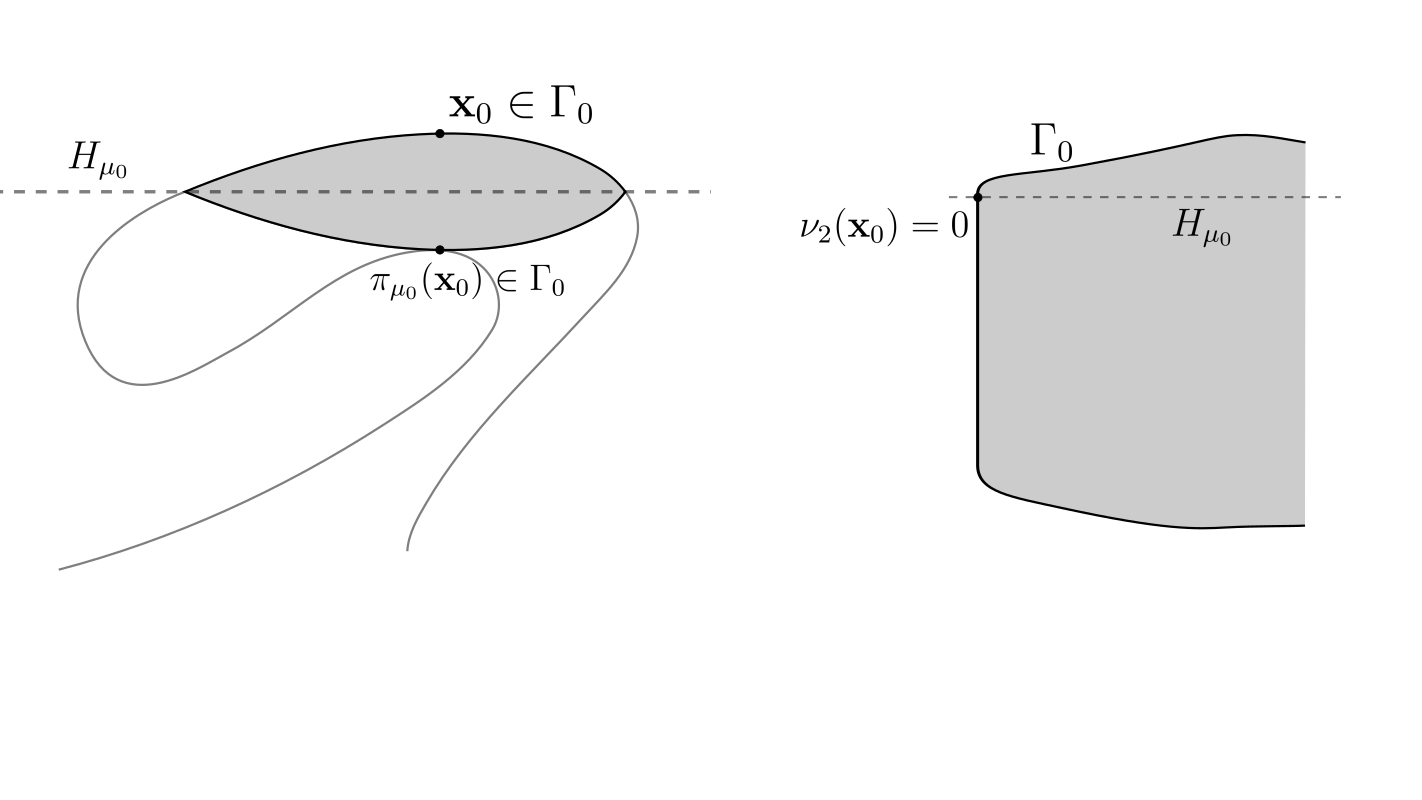}
\caption{On the left an internal tangency is represented at $\x_0\in \Gamma_0$. The figure on the right shows orthogonality of $\Gamma_0$ and $H_{\mu_0}$ at $x_0\in \Gamma_0$. }
\end{figure}

Hence, it must be $\mu_0 = \mu_m$. If $a_0\neq 0$, this shows that $-\partial_y\psi = \nu_2>0$ in $\Gamma_0$, since $\nu = -\nabla\psi$ on $\Gamma_0$. Therefore, the implicit function theorem applied to $\psi(x,y) = 0$ gives the existence of a global $C^2$ function $h_0(x)$ such that $\psi(x,h_0(x))=0$, that is, $\Gamma_0= \lbrace (x, h_0(x):\, x\in \mathbb{T}\rbrace $, where $h_0(x_0)=y_0$, for some $(x_0,y_0)\in \Gamma_0$ and $h_0(x)$ can be implicitly defined by the differential equation
\begin{equation}
\partial_x\psi(x,h_0(x)) + h_0'(x)\partial_y\psi(x,h_0(x))= 0.
\end{equation}  
{In case $a_0=0$, since $f(0)\neq 0$ we appeal to Lemma \ref{lemma:streamlineregularity} to write $\Gamma_0 = \lbrace \gamma_0(s): \, s\in \R\rbrace$, for some regular $C^2$ curve $\gamma_0(s)= (\gamma_0^1(s), \gamma_0^2(s))$. Then, $\hat\tau = \frac{\dot{\gamma}_0}{|\dot{\gamma}_0|}$ denotes the unit tangent vector of $\Gamma_0$ and $\hat n = \frac{\dot{\gamma}_0^\perp}{|\dot{\gamma}_0|}$ its exterior unit normal vector. Now, $\nu_2>0$ on $\Gamma_0$ reads $\dot{\gamma}_0^1 > 0$ and thus the inverse function theorem shows that $\gamma_0$ can be re-parametrized, that is, $\Gamma_0 = \lbrace (x,h_0(x)):\, x\in \mathbb{T}\rbrace$, with $h_0(x)=\gamma_0^2\circ (\gamma_0^1)^{-1}(x)$, rendering a $C^2$ graph.}
To show that $\Gamma_1$ is also graph-like, we just need to consider a moving plane argument starting from below, and moving upwards until the highest point of $\Gamma_1$, we omit the details.
\end{proof}

\subsection{Sliding Method for Two Free boundaries}
Now that we know the free-boundaries are given by $C^2$ graphs in the periodic variable $x\in \mathbb{T}$, we implement a sliding method to conclude that solutions must be invariant by $x$-translations, hence independent of $x$. We recall the fluid domain $D$ is given by
\begin{equation}
D= \lbrace (x,y):\, x\in\mathbb{T}, \, h_1(x)< y < h_0(x) \rbrace
\end{equation}
Here, $\Gamma_1 = \lbrace (x,h_1(x)) \rbrace_{x\in\mathbb{T}}$ and $\Gamma_0 = \lbrace (x,h_0(x)) \rbrace_{x\in\mathbb{T}}$ denote the boundaries of $D$, and we assume at least one of them is not identically constant. Otherwise, Proposition \ref{prop:rigiditychannel} already ensures that solutions must correspond to shear flows. 

\begin{proposition}
Let $\psi$ be a solution to  
\begin{equation}
\begin{cases}
\Delta\psi = f(\psi), \quad  0<\psi<1, & \text{ in }D=\Omega_{\Gamma_0} \setminus \overline{\Omega_{\Gamma_1}}, \\
\psi = 0, \quad |\nabla\psi| = a_0>0, & \text{ on }\Gamma_0= \lbrace (x,h_0(x)) \rbrace_{x\in\mathbb{T}}, \\
\psi = 1, \quad |\nabla\psi| = a_1, & \text{ on }\Gamma_1=\lbrace (x,h_1(x)) \rbrace_{x\in\mathbb{T}},
\end{cases}
\end{equation}
for some $f\in C([0,1])\cap C^1(0,1)$. Assume $f(1)\neq 0$ if $a_1=0$ and that $D$ contains no singular streamlines of $\psi$. Then, $\psi = \psi(y)$. 
\end{proposition}

\begin{proof}
The sliding method we implement must adapt to the free boundaries. For this, let $S_i=\Vert h_i' \Vert_{L^\infty}$ and choose a direction $\xi=(\xi_1,\xi_2)$ with $\xi_2>0$, $\xi_1\neq 0$ and such that $\frac{\xi_2}{\xi_1}\in (-S_0,S_0)\cup (-S_1,S_1)$, which is non-empty, by assumption. Next, we define several key quantities. Firstly, 
\begin{equation}
\tau_0:=\inf \lbrace \tau>0 : h_0(x+\mu\xi_1)-\mu\xi_2 \leq h_1(x), \text{ for all } x\in\mathbb{T}, \text{ for all } \mu\geq \tau \rbrace,
\end{equation}
so that $\tau_0>0$ denotes the first time $\tau>0$ for which the translated upper boundary $\Gamma_0-\tau\xi$ enters the fluid domain $D$. This is going to be our 'process starting time'. Similarly, we define
\begin{equation}
\tau_{h_0}:=\inf \lbrace \tau>0 : h_0(x+\mu\xi_1)-\mu\xi_2 < h_0(x), \text{ for all } x\in\mathbb{T}, \text{ for all } \mu\geq \tau \rbrace,
\end{equation}
and 
\begin{equation}
\tau_{h_1}:=\inf \lbrace \tau>0 : h_1(x+\mu\xi_1)-\mu\xi_2 < h_1(x), \text{ for all } x\in\mathbb{T}, \text{ for all } \mu\geq \tau \rbrace,
\end{equation}
they denote the first times $\tau_{h_0}\geq 0$ and $\tau_{h_1}\geq 0$ for which the translated curves $\Gamma_0-\tau\xi$ and $\Gamma_1-\tau\xi$ intersect $\Gamma_0$ and $\Gamma_1$, respectively. We claim that $\tau_{h_i}>0$ whenever $S_i\neq 0$. This is a consequence of the sliding direction having a non-zero angle within the range of the slope of the boundaries. Indeed, assume that for all $\tau>0$ and all $x\in\mathbb{T}$ we have 
\begin{equation}
h_{i}(x+\tau\xi_1) < h_{i}(x) +\tau\xi_2,
\end{equation}
that is
\begin{equation}
\frac{h_{i}(x+\tau\xi_1)- h_{i}(x)}{\tau\xi_1}< \frac{\xi_2}{\xi_1}, \text{ if } \xi_1>0
\end{equation}
and
\begin{equation}
\frac{h_{i}(x+\tau\xi_1)- h_{i}(x)}{\tau\xi_1}> \frac{\xi_2}{\xi_1}, \text{ if } \xi_1<0.
\end{equation}
Taking the limit as $\tau\xi_1\rightarrow 0$, we obtain
\begin{equation}
h'(x)\leq \frac{\xi_2}{\xi_1} \text{ if } \xi_1>0, \quad h'(x)\geq \frac{\xi_2}{\xi_1} \text{ if } \xi_1<0,
\end{equation}
for all $x\in \mathbb{T}$, that is, $\left|\frac{\xi_2}{\xi_1}\right|\geq S_i$, thus reaching a contradiction. Being the first time they intersect, the original and translated graphs are tangent at the intersection point.
We further define $\tau_1 := \max \lbrace \tau_{h_0}, \tau_{h_1} \rbrace>0$. The translated domain is given by
\begin{equation}
D_{\xi}^\tau := \lbrace \x=(x,y): x\in \mathbb{T}, \quad h_1(x+\tau\xi_1)< y + \tau\xi_2 < h_0(x+\tau\xi_1) \rbrace = D-\tau\xi, 
\end{equation}
which is a non-empty open set for all $\tau\in(\tau_1,\tau_1+\delta)$, for some $\delta>0$ small enough. It denotes all points $\x$ such that $\x+\tau\xi\in D$. 

\begin{figure}[h!]
\centering
\includegraphics[width=0.8\textwidth]{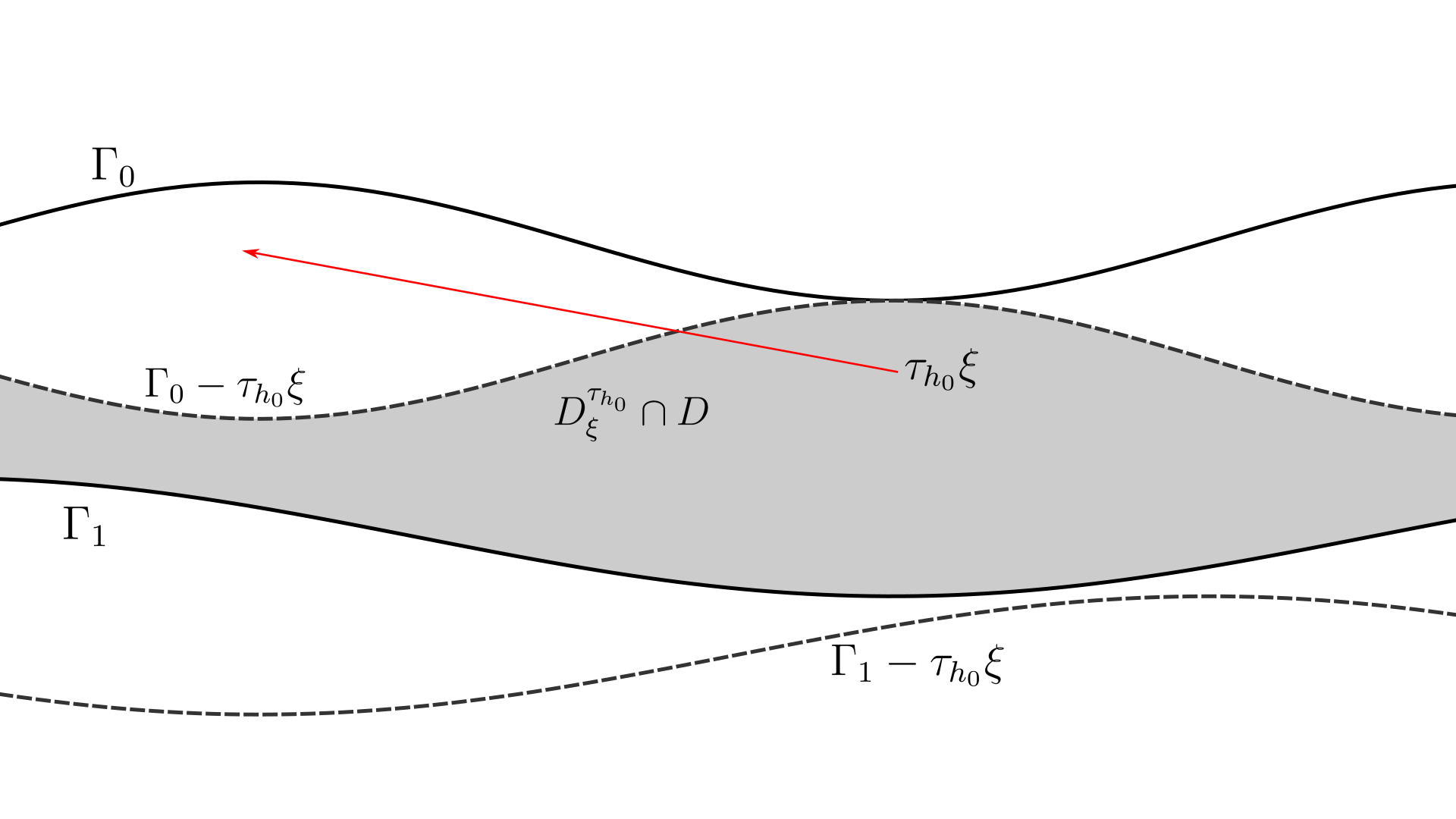}
\caption{The domain $D$ is translated by $-\tau_{h_0}\xi$. By definition, $\Gamma_0$ and $\Gamma_0-\tau_{h_0}\xi$ have non-empty intersection and $D^{\tau_{h_0}}_\xi \cap D $ denotes the region bounded above by $\Gamma_0-\tau_{h_0}\xi$ and below by $\Gamma_1$.}
\end{figure}

Next, let
\begin{equation}
\tau_2 := \sup \lbrace \tau>\tau_1: D^\mu_\xi \cap D \neq \emptyset, \text{ for all } \mu\in (\tau_1,\tau) \rbrace
\end{equation}
and observe that $\tau_2\leq \tau_0$ and $D^{\tau_2}_\xi \cap D=\emptyset$. By definition, $D^\tau_\xi \cap D \neq \emptyset$, for all $\tau\in (\tau_1,\tau_2)=:I$.
Now, let
\begin{equation}
J=\lbrace \tau\in I : w^\tau(\x):= \psi(\x+\tau\xi) - \psi(\x) < 0, \quad\forall\x\in \overline{D_{\xi}^\tau\cap D} \rbrace 
\end{equation}
Firstly, $J$ is nonempty. This is the so-called initial step in the comparison principle within the sliding method, see \cite[Theorem 3]{Reichel}. To see this, for some $\varepsilon>0$ small, consider the interval $(\tau_2-\varepsilon, \tau_2)$, in which $D_\xi^\tau\cap D\neq \emptyset$. By uniform continuity, there exists $\delta>0$ such that $1\geq\psi(\x)> 9/10$, if $\text{dist}(\x,\Gamma_1)<\delta$ and $0\leq\psi(\x)<1/10$, if $\text{dist}(\x,\Gamma_0-\tau\xi)<\delta$. We claim there exists $\varepsilon>0$ sufficiently small such that $\text{dist}(\x,\Gamma_1)<\delta$ and $\text{dist}(\x+\tau\xi,\Gamma_0)<\delta$, for all $\x\in D^\tau_\xi\cap D$, for all $\tau\in (\tau_2-\varepsilon, \tau_2)$. 

If not, we can find a sequence of $\tau_n\nearrow \tau_2$ and $\x_n\in D_\xi^{\tau_n}\cap D$ such that either $\text{dist}(\x_n,\Gamma_1)\geq\delta$ or $\text{dist}(\x_n+\tau\xi,\Gamma_0)\geq\delta$. We assume the former takes place. Up to a subsequence, $\x_n\rightarrow\x_0=(x_0,y_0) \in \overline{D_\xi^{\tau_2}\cap D}$, we shall see that $D_\xi^{\tau_2}\cap D$ is non-empty, thus contradicting the definition of $\tau_2$. Indeed, since $\x_0\in \overline{D}$, and $\text{dist}(\x_n,\Gamma_1)\geq\delta$, there holds $\text{dist}(\x_0,\Gamma_1)\geq\delta/2$ and thus $h_1(x_0) < y_0 - \epsilon$, for all $\epsilon >0$ sufficiently small. On the other hand, we have that $y_0\leq h_0(x_0)$ and $y_0 + \tau_2\xi_2 \leq h_0(x_0+\tau_2\xi_1)$, so that $y_0-\epsilon < h_0(x_0)$ and $y_0 - \epsilon < h_0(x_0 + \tau_2\xi_1) -\tau_1\xi_2$, for $\epsilon>0$. Hence, $(x_0,y_0-\epsilon)\in D_\xi^{\tau_2}\cap D$, it is non-empty. Therefore, for some $\varepsilon>0$ sufficiently small there holds $\text{dist}(\x,\Gamma_1)<\delta$ and $\text{dist}(\x+\tau\xi,\Gamma_0)<\delta$, for all $\x\in D^\tau_\xi\cap D$, for all $\tau\in (\tau_2-\varepsilon, \tau_2)$ and we conclude that
\begin{equation}
w^\tau(\x) = \psi(\x+\tau\xi) - \psi(\x) \leq 1/10 - 9/10  = -4/5<0.
\end{equation}
so that $w^\tau< 0$ in $\overline{D^\tau_\xi\cap D}$, for all $\tau\in (\tau_2-\varepsilon, \tau_2)$.

We now claim that $\tau_1=\inf J$. Assume otherwise and let $\alpha:=\inf J >\tau_1$. By definition of infimum, $w^\alpha \leq 0$ in $\overline{D_\xi^\alpha \cap D}$ and there must be some ${\x_0}\in \overline{D_\xi^\alpha \cap D}$ such that $w^\alpha({\x_0}) = 0$. Moreover, due to the boundary conditions, we can show that $w^\alpha(\x)<0$, for all ${\x}\in \partial (D_{\xi}^\alpha\cap D)$. Indeed, if $\x=(x,y)\in \partial (D_{\xi}^\alpha\cap D)$, and since $\alpha>\tau_1$, we must have either
\begin{enumerate}
\item $\x\in\Gamma_1$, $\x+\alpha\xi\not\in\Gamma_1$ and $y+\alpha\xi_2 \leq h_0(x+\alpha\xi_1)$, or 
\item $\x+\alpha\xi\in\Gamma_0$, $\x\not\in\Gamma_0$ and $y\geq h_1(x)$.
\end{enumerate}
In both cases, since $\psi = i$ on $\Gamma_i$ and $D=\lbrace 0 < \psi < 1\rbrace$, there holds $w^\tau(\x) < 0$. Here it is crucial the fact that $\alpha>\tau_1$, which ensures that $\Gamma_0 -\alpha\xi$ and $\Gamma_1 - \alpha\xi$  do not intersect with $\Gamma_0$ and $\Gamma_1$, respectively. We thus conclude that ${\x_0}$ is an interior point of $D_{\xi}^\alpha\cap D$. We can then find a ball $B_r({\x}_0)\subset D$ such that $B_r({\x}_0+\alpha\xi) \subset D$ as well, for some $r>0$ sufficiently small. Since $w^\alpha(\x)\leq 0$ in $B_r(\x_0)$, $w^\alpha(\x_0)= 0$ and $f$ is of class $C^1$ in $B_r(\x_0)$, by the maximum principle we conclude that $w^\alpha\equiv 0$ in $B_r(\x_0)$. By unique continuation, we extend $w^\alpha\equiv 0$ in $D_{\xi}^\alpha\cap D$ and, by continuity, in $\overline{D_{\xi}^\alpha\cap D}$. This contradicts the fact that $w^\alpha<0$ on $\partial (D_{\xi}^\alpha\cap D)$, resulting in $\alpha=\tau_1$.

Once $\inf J = \tau_1$ is established, we likewise have that $w^{\tau_1}\leq 0$ in $\overline{D_{\xi}^{\tau_1}\cap D}$. In fact, by definition of $\tau_1$, there is some $\x_0=(x_0,y_0)\in \partial (D_{\xi}^{\tau_1}\cap D)$ such that either $\x_0\in\Gamma_0\cap(\Gamma_0-\tau_1\xi)$ or $\x_0\in\Gamma_1\cap (\Gamma_1-\tau_1\xi)$. We assume the former occurs, namely that $\tau_1=\tau_{h_0}$. Now, since $\tau_1=\inf I$, the curves $\Gamma_0$ and $\Gamma_0-\tau_1\xi$ intersect tangentially at $\x_0=(x_0,y_0)$, namely $h_0'(x_0) = h_0'(x_0+ \tau_0\xi_1)$. By the overdetermined boundary condition $\partial_\nu\psi=a_0$ on $\Gamma_0$, there holds
\begin{equation}
\partial_{\nu}w^{\tau_1}(\x_0) = 0,
\end{equation}
where $\nu$ denotes the outer normal unit vector to $\Gamma_0$ at $\x_0$. We now proceed as in the proof of Lemma \ref{lemma:graphlikeboundary}. We can place a small ball $B_r(\x_1)$ on $D_{\xi}^{\tau_1}\cap D$ tangent to $\Gamma_0$ at $\x_0$ such that $B_r+\tau_1\xi$ is again tangent to $\Gamma_0$ at $\x+\tau_1\xi$. If $a_0\neq 0$, then $f$ is of class $C^1$ there and the Hopf Lemma ensures that $w^{\tau_1}\equiv 0$ in $B_r(\x_1)$. If $a_0=0$, we then have $f(0)\neq 0$ and we apply the Hopf Lemma for singular operators as in the proof of Lemma \ref{lemma:graphlikeboundary} to obtain $w^{\tau_1}\equiv 0$ in $B_r(\x_1)$ as well, we omit the details. 

We extend $w^{\tau_1}\equiv 0$ first to $D_{\xi}^{\tau_1}\cap D$ and then to $\overline{ D_{\xi}^{\tau_1}\cap D}$. However, we can find some $\x\in\overline{ D_{\xi}^{\tau_1}\cap D}$ with $w^{\tau_1}(\x)< 0$. Indeed, let $\x_0=(x_0,y_0)\in \Gamma_0 \cap (\Gamma_0 - \tau_1\xi)$, that is, $h(x_0) = y_0 = h(x_0+\tau_1\xi_1) - \tau_1\xi_2$, let 
\begin{equation}
I^{\tau_1} := \lbrace x\in \mathbb{T}: h_0(x) = h_0(x+\tau_1\xi_1) - \tau_1\xi_2 \rbrace
\end{equation}
and let $I_{x_0}\subset I^{\tau_1}$ denote the maximal connected component of $I^{\tau_1}$ that contains $x_0$, we write $I_{x_0} = [x_{l},x_r]$, with $x_l \leq x_0 \leq x_r$. By periodicity, $\mathbb{T}\setminus I_{x_0}\neq \emptyset$, and so there exists some $\delta_0>0$ such that $(x_r,x_r+\delta)\cap I^{\tau_1}= \emptyset$, that is,
\begin{equation}
h(x_r + \delta+\tau_1\xi_1) - \tau_1\xi_2 < h(x_r+\delta),
\end{equation}
for all $0<\delta<\delta_0$. On the other hand, we have that 
\begin{equation}
h_0(x_r + \delta + \tau_1\xi_1) - \tau_1\xi_2 \geq h_0(x_r+\tau_1\xi_1) -\tau_1\xi_2 - \Vert h_0' \Vert_{L^\infty} \delta_0 = h_0(x_r) - \Vert h' \Vert_{L^\infty}\delta_0
\end{equation}
while
\begin{equation}
h_1(x_r+\delta) \leq h_1(x_r) + \Vert h_1' \Vert_{L^\infty} \delta_0,
\end{equation}
so that 
\begin{equation}
h_0(x_r + \delta + \tau_1\xi_1) - \tau_1\xi_2 > h_1(x_r+\delta)
\end{equation}
holds as long as
\begin{equation}
\delta_0 < \frac{h_0(x_r) - h_1(x_r)}{\Vert h_0' \Vert_{L^\infty} + \Vert h_1' \Vert_{L^\infty}}.
\end{equation}
We remark here that, by assumptions, at least one of the $h_i$ has $\Vert h_i'\Vert_{L^\infty}>0$. Thus, for $\delta_0>0$ small enough and $0<\delta<\delta_0$, we let $\x=(x,y)$ with $x=x_r+\delta$ and $y=h_0(x+\tau_1\xi_1)-\tau_1\xi_2$. It is such that $\x\in\overline{D_\xi^{\tau_1}\cap D}$ because $y\in (h_1(x),h_0(x))$ and
\begin{equation}
w^{\tau_1}(\x) = \psi(\x+\tau_1\xi) - \psi({\x_1}) =  - \psi(\x) < 0,
\end{equation}
because $\x\in \overline{D}\setminus\Gamma_0$ and $\x\in\Gamma_0-\tau_1\xi$. This contradicts $w^{\tau_1}\equiv 0$ in $\overline{D_\xi^{\tau_1}\cap D}$.

To summarize, we reach a contradiction whenever we choose sliding directions whose non-zero angle falls within the range of $h_0'$ and $h_1'$. Therefore, it must be that  $h_i'\equiv 0$, for $i=0,1$, thus obtaining completely horizontal fluid boundaries. The proof is completed once we appeal to Proposition \ref{prop:rigiditychannel}.
\end{proof}

\section{Allowing for open sets of zero velocity}\label{zeroregion}
The assumptions on the level lines of the stream-function $\psi$ can be relaxed a bit. For instance, albeit with some simple modifications that we describe below, the proof adapts to the following assumptions.

For any $c$ in the range of $\psi$, we denote $A_c$ to be a maximally connected component of $\psi^{-1}(c)$. We say that $A_c$ is 
\begin{itemize}
\item a \emph{regular level set} if $\nabla\psi \neq 0$ on $A_c$.
\item a \emph{regular-singular level set} if there exists $p,q\in A_c$ such that $\nabla\psi(p) = 0$ and $\nabla \psi(q) \neq 0$.
\item a \emph{singular level set} if $\nabla\psi(p)=0$, for all $p\in A_c$.
\end{itemize}
By the implicit function theorem, a regular level set $A_c$ can be locally parametrized by a $C^2$ curve, and the same holds for points in regular-singular level sets that have non-trivial gradient. We next make some basic assumptions on the geometry of these level sets. We shall assume that any $A_c$ is given by the closed set $\overline{\Omega_{\Gamma_{c,\mathrm{out}}}}\setminus \Omega_{\Gamma_{c,\mathrm{in}}}$, where $\Gamma_{c,\mathrm{out}}$ and $\Gamma_{c,\mathrm{in}}$ are closed non-contractible curves. While we do not assume them to be disjoint, we instead require that the open bounded domains they define through the Jordan Curve Theorem satisfy the inclusion relation ${\Omega_{\Gamma_{c,\mathrm{out}}}}\supseteq \Omega_{\Gamma_{c,\mathrm{in}}}$, so that $A_c$ is non-empty and $\partial A_c= \Gamma_{c,\mathrm{out}}\cup \Gamma_{c,\mathrm{in}}$, possibly with $\Gamma_{c,\mathrm{out}}\cap \Gamma_{c,\mathrm{in}}\neq \emptyset$. In fact, if $\Gamma_{c,\mathrm{out}}= \Gamma_{c,\mathrm{in}} =:\Gamma_c$, then $A_c=\Gamma_c$, the level set is a closed non-contractible curve. Since $\nabla\psi\neq 0$ on regular level sets $A_c$, there holds $A_c=\Gamma_c$.

\begin{theorem}\label{thm:levelsetrigidity}
Let $M$ be the straight periodic channel  or the circular annulus. Let $u\in C^2(M)$ be a stationary solution of the Euler equations in $M$.  Suppose that all maximally connected level sets of the stream-function of $u$ are topologically equivalent to non-contractible loops.  Moreover assume  each connected component of the boundary satisfies either that they 
\begin{itemize}
\item are maximally connected subsets of their respective level sets, or
\item  belong to a singular level set, e.g. they have zero velocity.
\end{itemize}
 Then, $u$ is a shear flow or a circular flow respectively.
\end{theorem}

In what follows, we comment on the main modifications. To begin with, Lemma \ref{lemma:streamlines} now reads

\begin{lemma}
Boundary points of both regular-singular and singular level sets are not isolated from regular streamlines. 
\end{lemma}

\begin{proof}
Assume there is some boundary point isolated from regular streamlines. Any such boundary point $p$ of a regular-singular level set $A_c$ is such that $B_r(p)\cap A_c^c\neq \emptyset$, for all $r>0$. Here, $A_c^c$ denotes the complement of $A_c$. As before, if $\nabla\psi(p)\neq 0$ we have that $\nabla\psi\neq 0$ on $B_r(p)$, for some $r>0$ sufficiently small. By assumption, all $q\in B_r(p)$ do not belong to regular streamlines, neither to singular level sets because they have non-zero gradient, Thus they must lie on regular-singular streamlines. Flowing the ball $B_r(p)$ by $u=\nabla^\perp \psi$, we compress the ball, a contradiction.

On the other hand, if $\nabla\psi(p) = 0$, we still have that $\nabla\psi\not \equiv 0$ on $B_r(p)$. Indeed, if that where the case, we would have $\psi \equiv c$ on $B_r(p)$, and thus $B_r(p) \subset A_c$, a contradiction with $p$ being a boundary point. Thus, there is some $q\in B_r(p)$ with $\nabla\psi(q)\neq 0$. For some $\delta>0$ small enough, we have that $\nabla\psi\neq 0$ on $B_\delta(q) \subset B_r(p)$, namely all points in $B_\delta(q)$ belong to a regular-singular level sets and have non-trivial velocity there. As before, flowing the ball $B_\delta(q)$ we obtain a contradiction with incompressibility.

Hence, boundary points of regular-singular level sets cannot be isolated from regular streamlines. The arguments for boundary points of singular level sets are the same, we omit the details.
\end{proof}

The construction of the fluid sub-domain in Proposition \ref{prop:fluiddomain} remains unchanged. We comment on the main modifications required. Firstly, assuming that $\partial_{\hat n}\psi < 0$ on a regular level set $\Gamma_{c_0}$, the definitions of $c_-$ and $c_+$ now read
\begin{align}
c_- := \inf \lbrace \psi(A_c) : A_c &= \overline{\Omega_{\Gamma_{c,\mathrm{out}}}}\setminus \Omega_{\Gamma_{c,\mathrm{in}}} \text{ is a level set},\,  \psi|_{A_c}=c \text{ and } \\
&\quad\,\, \Omega_{\Gamma_{c,\mathrm{in}}}\setminus \overline{\Omega_{\Gamma_{c_0}}} \neq \emptyset  \text{ contains no singular level sets} \rbrace
\end{align}
and likewise
\begin{align}
c_+ := \sup \lbrace \psi(A_c) : A_c &= \overline{\Omega_{\Gamma_{c,\mathrm{out}}}}\setminus \Omega_{\Gamma_{c,\mathrm{in}}} \text{ is a level set},\, \psi|_{A_c}=c \text{ and }\\
&  \Omega_{\Gamma_{c_0}}\setminus \overline{\Omega_{\Gamma_{c,\mathrm{out}}}} \neq \emptyset  \text{ contains no singular level sets} \rbrace.
\end{align}

\begin{proof}[Proof of Proposition \ref{prop:fluiddomain}]
The infimum $c_-$ is attained at a level set associated to $c_-$, which is achieved by continuity and compactness of a sequence of approximating level sets. More precisely, the compactness argument produces gives a point $p$, and we define $A_c\ni p$ to be the maximally connected component of the level set in which $p$ lies. 

Now, if $A_c$ is interior to $M$, then it is a singular level, the argument to prove so remains unchanged: any $q\in A_c$ with $\nabla\psi(q)\neq 0$ must, not only be a boundary point, namely $q\in \Gamma_{c,\mathrm{out}}$ or $q\in \Gamma_{c,\mathrm{in}}$, but also $q\not \in \overline{\text{int} A_c}$, that is $q\in \Gamma_{c,\mathrm{out}}\cap \Gamma_{c,\mathrm{in}}$. There, if $\partial_{\hat n}\psi(q)>0$ we reach a contradiction with $c_- < \psi < c_0$ in $D$, while if $\partial_{\hat n}\psi(q)<0$, then $A_c$ is not singular and thus $c_-$ is not the infimum.

On the other hand, if $A_c$ has non-trivial intersection with the upper channel boundary, then by assumption $A_c$ is either singular or the upper boundary itself.
\end{proof}

We next present a suitable adaptation to Lemma \ref{lemma:mainellipticeq}. In what follows, $D$ is defined to be the region bounded by the limiting level sets above and below $\Gamma_{c_0}$. By simplicity, we take $D$ to be the region bounded above by $\Gamma_{c_-, \mathrm{in}}$ and bounded below by $\Gamma_{c_+, \mathrm{in}}$, although the other three choices are equally valid.

\begin{proof}[Proof of Lemma \ref{lemma:mainellipticeq}]
The absence of singular level sets prevents the appearance of  local maxima and minima. Hence, $\psi$ attains different values on its boundaries. The argument to show that $f$ is well-defined on regular streamlines and regular-singular level sets given by closed non-contractible curves requires no modification. For a regular-singular level set with non-empty interior, it is clear that the vorticity is zero on the closure of its interior. All other points of the regular-singular level set are boundary points that are not limits of interior points. The arguments in the proof show that $f$ is well defined as well for them, namely the vorticity is constant on these points. Since the regular-singular level set is connected and the vorticity is continuous, then the whole regular-singular level set has zero vorticity, $f$ is well defined. The continuity of $f$ needs no adjustments, neither does its $C^1$ smoothness, although we remark here again that the existence of points $q$ in regular-singular level sets $A_c$ with $\nabla\psi(q)\neq 0$ is crucial to derive a $C^1$ smooth mapping of stream values to vorticity values.
\end{proof}

Once these issues have been addressed, most of the arguments in Sections \ref{sec:rigidoverdetchannel}, \ref{sec:rigidoverdetannulus} and \ref{sec:rigidfreeboundary} apply verbatim to the (over-determined) elliptic problems obtained in Lemma \ref{lemma:mainellipticeq}. However, the reasoning concerning regular-singular level sets on Proposition \ref{prop:rigidonefree} and Proposition \ref{prop:rigidityannulus} needs to be modify. We argue for the annular case, the idea applies equally to the channel scenario.

\begin{proof}[Proof of Proposition \ref{prop:rigidityannulus}]
We only need to study regular-singular level sets. Assume $A_c$ is such a regular-singular level set and let $q\in\Gamma_c$ be such that $\nabla\psi(q)\neq 0$. Then, $q\in U_k$ for some $k\geq 1$, so that $r_k < |q-\x_k|<R_k$. Therefore, the set $\Gamma := \lbrace \x\in M: |x-\x_k| = |q-\x_k|\rbrace$ is a closed non-contractible circle, contained in $U_k$ and, more importantly, it is such that $\psi(\x) = \psi(q) =c$, for all $\x\in \Gamma$. Assume now that there is some $q_0\in A_c\setminus \Gamma$, that is, $|q_0-\x_k| > |q-\x_k|$ or  $|q_0-\x_k| < |q-\x_k|$. Since the $U_k$ are nested, $\psi$ is radially decreasing there and possibly flat in $S$, in the former case we reach $\psi(q_0) > \psi(q)$, while in the latter case we obtain $\psi(q_0) < \psi(q)$. Both lead to $q_0\not\in A_c$, a contradiction. Hence, it must be $A_c=\Gamma$ and $\partial_\rho\psi < 0$ in $A_c$, thus concluding that there are no regular-singular level sets.
\end{proof}
The channel scenario is equally similar. Sliding a point in a regular-singular level set with $\partial_y\psi < 0$ we obtain a closed non-contractible curve $\Gamma$ such that $\partial_y\psi < 0$ there. Since $\partial_y\psi \leq 0$ in $\mathbb{T}\times(0,1)$ we then must have $A_c=\Gamma$ as before, since any point $p$ not in $\Gamma$ must necessarily have  $\psi(p) < c$ or $\psi(p)> c$.

%
%
%
%
%
%
%

\section{Rigidity of shear flows: Proof of Theorem \ref{mainshearrigid}}\label{sec:mainshearrigid}
After having obtained the above rigidity results for solutions to free boundary problems we are now in position to prove Theorem \ref{mainshearrigid}
\begin{proof}[Proof of Theorem \ref{mainshearrigid}]
The proof studies local properties of the background shear flow according to the vanishing of the vorticity on the set of stagnation points, in order to conclude that nearby stationary solutions are laminar. We shall see that if the vorticity is non-zero, we can find a small tubular region where sufficiently close velocity fields $u_\varepsilon$ only vanish along an unique streamline, so that $u_\varepsilon$ is laminar inside the tubular domain. If the background vorticity is zero, then we can construct a fluid sub-domain within a small tubular region where (1) the vorticity has no critical points there, (2) the vorticity vanishes on a closed non-contractible curve inside and (3) any stagnation points of $u_\varepsilon$ must lie within this fluid sub-domain, so that $u_\varepsilon$ is also laminar there. Throughout the proof we assume that $u_\varepsilon$ is a $C^2(\overline{M})$ steady solution of the Euler equations such that $\Vert u_\varepsilon - u \Vert_{C^2}=\varepsilon$, where $u=(U(y),0)$. For convenience, we assume that $U(0)=0$.
\subsection*{Step 1}Assume first that $U'(0)=U_0'\neq 0$. Then, $U(y) = \int_0^y U'(s)  \rmd s \geq c_0y$, for all $\delta_1 \geq y\geq 0$ and there holds $|U(y)|\geq c_0\delta_0$, for all $\delta_1 \geq |y|\geq \delta_0>0$. Moreover, by continuity, $|U''(y)|\geq c_1>0$, for all $|y|\leq \delta_1$, for some $\delta_1>0$ sufficiently small.

Now, any sufficiently small in $C^2$ perturbation $u_\varepsilon(x,y)$ of the steady velocity field $(U(y),0)$ is such that $|\partial_y\omega_\varepsilon|>\frac{c_1}{4}$, for all $|y|\leq \delta_1$, for $\varepsilon>0$ sufficiently small and independent of $\delta_1$. On the other hand, $|u^1_\varepsilon(x,y)|\geq \frac{c_0}{2}\delta_0$, for all $|y|\geq \delta_0$ and all $\varepsilon< \frac{c_0}{2}\delta_0$. Additionally, due to the strict monotonicity of $u^1_\varepsilon$, there exists a curve $\Gamma\subset \mathbb{T}\times(-1,1)$ such that $u^1_\varepsilon(x,y)=0$, for all $(x,y)\in \Gamma$. In fact, by the previous observation, such curve $\Gamma$ is contained in $\mathbb{T}\times (-\delta_0,\delta_0)$ and $u^1_\varepsilon>0$ in the region bounded below by $\Gamma$ and $u^1_\varepsilon<0$ in the region bounded above by $\Gamma$.

Taking $\delta=\min\lbrace \frac{c_0}{2}\delta_0, \delta_1 \rbrace$, we see that $u^1_\varepsilon$ vanish in $\Gamma\subset \mathbb{T}\times (-\delta, \delta)$ and $|\partial_y\omega_\varepsilon| > 0$  in $\Gamma\subset \mathbb{T}\times (-\delta, \delta)$. In particular, from the Euler equations
\begin{equation}
0=u^1_\varepsilon\partial_x\omega_\varepsilon + u^2_\varepsilon\partial_y\omega_\varepsilon
\end{equation}
we conclude that $u^2_\varepsilon\equiv0$ on $\Gamma$ as well. Hence, in $\mathbb{T}\times (-\delta,\delta)$, we see that $u_\varepsilon$ only vanishes on the streamline $\Gamma$. 

\subsection*{Step 2} Assume now that $U'_0=0$ and $U_0'':=U''(0)>0$, say. Since $U\in C^2$, there exists some $\delta_0>0$ such that $\frac{U_0''}{2}\leq U''(y) \leq \frac{3U_0''}{2}$ in $(-\delta_0,\delta_0)$ and there holds 
\begin{equation*}
- \frac{3U_0''}{2}y \leq \omega(y) \leq -\frac{U_0''}{2}y, 
\end{equation*}
for all $y\in (0,\delta_0)$ and
\begin{equation*}
-\frac{U_0''}{2}y \leq \omega(y) \leq -\frac{3U_0''}{2}y, 
\end{equation*}
for all $y\in (-\delta_0,0)$. Since $\partial_y\omega(y)=-U''(y) < -\frac{U_0''}{2}$ any sufficiently small $C^2$ perturbation $u_\varepsilon$ of $u$ must have  $\partial_y\omega_\varepsilon \leq -\frac{U_0''}{4}$ hence $\nabla\omega_\varepsilon\neq 0$ in $\mathbb{T}\times(-\delta_0,\delta_0)$. We now find a fluid sub-domain in $\mathbb{T}\times(-\delta_0,\delta_0)$. Let $0<\delta_1 < \delta_2 < \delta_0$, such that $\delta_1 < \frac{1}{8}\delta_2$ and $\delta_2 < \frac{1}{10}\delta_0< \frac{1}{8}\delta_0$. Then,
\begin{equation}
|\omega(y)| <  \frac{3}{2}U_0''\delta, \text{ for all } |y|<\delta
\end{equation}
and 
\begin{equation}
 \frac{U_0''}{2}\delta_i < |\omega(y)| <  \frac{3}{2}U_0''\delta_j, \text{ for all } \delta_i < |y| <\delta_j.
\end{equation}
Therefore, any $\omega_\varepsilon$ sufficiently close to $\omega$ satisfies
\begin{equation}
|\omega_\varepsilon(x,y)| <  2U_0''\delta, \text{ for all } |y|<\delta
\end{equation}
and 
\begin{equation}
 \frac{U_0''}{4}\delta_i < |\omega_\varepsilon(x,y)| <  2U_0''\delta_j, \text{ for all } \delta_i < |y| <\delta_j.
\end{equation}
Since $\delta_1 < \frac{1}{8}\delta_2$, we deduce that for any  $\x\in\mathbb{T}\times(-\delta_1, \delta_1)$ the vortex-line $\Gamma_{\omega_\varepsilon(\x)}\subset\mathbb{T}\times(-\delta_2,\delta_2)$ and likewise, any $\x\in\mathbb{T}\times(-\delta_0/10, \delta_0/10)$ has its associated the vortex-line $\Gamma_{\omega_\varepsilon(\x)}\subset\mathbb{T}\times(-\delta_0,\delta_0)$. More importantly, for all $\x=(x,y)$ such that $\delta_2 < |y| < \delta_0/10$, we have $|\omega_\varepsilon(\x)| > \frac{U_0''}{4}\delta_2$ and the associated vortex line $\Gamma_{\omega_\varepsilon(\x)}\subset \mathbb{T}\times(-\delta_0,\delta_0)$ is such that $\Gamma_{\omega_\varepsilon(\x)}\cap\mathbb{T}\times (-\delta_1,\delta_1)=\emptyset$, since $|\omega_\varepsilon|< 2U_0''\delta_1$ there.

Now, let $\x_1=(x_1,y_1)$, for some $x_1\in\mathbb{T}$ and $\delta_2 < y_1 < \delta_0/10$ and let $\x_2=(x_1,-y_1)$. If $\omega_\varepsilon$ is sufficiently close to $\omega$ in $C^1$, there holds $\omega_\varepsilon(\x_1) < \omega_\varepsilon(\x_2)$. Let $\Gamma_{\omega_\varepsilon(\x_i)}$ be the associated vortex-lines. Since $\Gamma_{\omega_\varepsilon(\x_i)}\subset \mathbb{T}\times (-\delta_0, \delta_0)$, we have $\nabla\omega_\varepsilon\neq 0$ on $\Gamma_{\omega_\varepsilon(\x_i)}$. Therefore, $\Gamma_{\omega_\varepsilon(\x_i)}$ are $C^1$ closed non-contractible curves. Since the normal vector on $\Gamma_{\omega_\varepsilon(\x_i)}$ is proportional to $\nabla\omega_\varepsilon\neq 0$ and $u_\varepsilon\cdot\nabla\omega_\varepsilon = 0$, we deduce that $u_\varepsilon$ is tangent to $\Gamma_{\omega_\varepsilon(\x_i)}$, and the region $D=\Omega_{\Gamma_{\omega_\varepsilon(\x_1)}}\setminus\overline{\Omega_{\Gamma_{\omega_\varepsilon(\x_2)}}}$ in $\mathbb{T}\times(-\delta_0,\delta_0)$ constitutes an admissible fluid sub-domain in which $u_\varepsilon$ is laminar. 

Indeed, from the Euler equations we now observe that
\begin{equation}
\nabla^\perp\omega_\varepsilon \cdot \nabla\psi_\varepsilon = 0
\end{equation}
in $D$, with $u_\varepsilon\cdot \hat{n}=0$ on $\partial D = \Gamma_{\omega_\varepsilon(\x_1)}\cup \Gamma_{\omega_\varepsilon(\x_2)}$. Now, since $\nabla\omega_\varepsilon\neq 0$, the vorticity effectively transports the stream-function along its vortex-lines. Therefore, by Lemma \ref{lemma:mainellipticeq}, there exists some $g\in C([\omega_\varepsilon(\x_1), \omega_\varepsilon(\x_2)])\cap C^1(\omega_\varepsilon(\x_1), \omega_\varepsilon(\x_2))$ such that $\psi_\varepsilon = g(\omega_\varepsilon)$ in $D$, so that the stream-lines of $\psi$ have the same topology as the vortex-lines of $\omega_\varepsilon$, which are closed non-contractible curves. Hence, $u_\varepsilon$ is laminar in $D$. Furthermore, since $U(y) \geq \frac{U_0''}{4}y^2$  in $(-\delta_0,\delta_0)$, we have that $U(y) \geq \frac{U_0''}{4}\delta_1^2$, for all $|y|\geq \delta_1$. Then, for sufficiently small perturbations $u_\varepsilon$ of $u$, there holds $u^1_\varepsilon \geq \frac{U_0''}{8}\delta_1^2$, for all $|y|\geq \frac{\delta_1}{4}$.

\subsection*{Step 3} We can repeat the above arguments for finitely many other stagnation lines $y=y_i$ of $U(y)$. For those $y_i$ with $U'(y_i)\neq 0$, we can find a small periodic tube $\mathbb{T}\times (y_i-\delta_i, y_i + \delta_i)$ where $u_\varepsilon=0$ only on a closed non-contractible curve $\Gamma_{i}$. For those $y_i$ with $U'(y_i) = 0$, we likewise find fluid sub-domains $D_i= \Omega_{\Gamma_i^-}\setminus \overline{\Omega_{\Gamma_i^+}}$ within a small periodic tube $\mathbb{T}\times (y_i-\delta_i, y_i + \delta_i)$ bounded by the vortex-lines (and also streamlines) $\Gamma_i^-$ and $\Gamma_i^+$. Also, $u_\varepsilon$ is laminar in $D_i$ and any possible stagnation point of $u_\varepsilon$ in $\mathbb{T}\times(y_i-\delta_i, y_i+\delta_i)$ lies inside the fluid domain $D_i$. Set then $I = (-1,1)\setminus \cup_{i=0}^m (y_i-\delta_i, y_i+\delta_i)$ and since $|U(y)| \geq c >0$ on $I$, we note that $u^1_\varepsilon\neq 0$, and thus $u_\varepsilon\neq 0$, on $\mathbb{T}\times I$. 

We now observe that $u_\varepsilon$ is laminar in $\mathbb{T}\times(-1,1)$. To see this, note that the streamlines $\Gamma_i$, $\Gamma_i^-$ and $\Gamma_i^+$ decompose the periodic channel into regions bounded by them. Between any two $\Gamma_i^-$ and $\Gamma_i^+$ we have seen that $u_\varepsilon$ is laminar, while since either $\Omega_{\Gamma_{i+1}}\setminus \overline{\Omega_{\Gamma_{i}^-}}$ or $\Omega_{\Gamma_{i+1}^+}\setminus \overline{\Omega_{\Gamma_{i}^-}}$ are fluid sub-domains (the boundaries are stream-lines of the flow, $u_\varepsilon$ is tangent there) and $u_\varepsilon\neq 0$ there, $u_\varepsilon$ is laminar as well. Theorem \ref{thm:rigidity} allows us to conclude that $u_\varepsilon$ is a shear flow.
\end{proof}

Finally, we end with a much stronger rigidity result assuming non-vanishing velocity. 
Consider the 2d Euler equation written for a perturbation $\omega$ about a shear flow $(v(y),0)$:
\begin{align}\label{euler}
\partial_t \omega + v(y) \partial_x \omega -v''(y) \partial_x \Delta^{-1} \omega&=  -u \cdot \nabla \omega.
\end{align}
with $u= K[\omega]$.
Under the non-vanishing assumption, this can be recast with $\psi= \Delta^{-1}\omega$ as
\be
\partial_t \omega+  \nabla \cdot \left(|v(y)|^2 \nabla \left(\frac{\partial_x \psi }{v(y)}\right)\right)  =-  u \cdot \nabla  {\omega}.
\ee
\begin{theorem}\label{rigiditynonstagnant}
Let $\bar{u}=(v(y),0)$ be a non-stagnant shear flow (e.g. satisfying $\inf_{\Omega} |v| >0$) and $\Omega$  be a periodic channel.  There exists an $\varepsilon:= \varepsilon(\inf_{\Omega} |v|, \|v\|_{C^1}, \Omega)$ such that all stationary Euler states in $u$ satisfying $\| u-\bar{u}\|_{W^{1, 2+}(\Omega)}< \varepsilon$ are shear flows.
\end{theorem}
\begin{proof} 
On time independence solutions ($\partial_t \omega=0)$, multiplying by $\frac{\partial_x \psi}{v(y)}$ and integrating  we find
\begin{align*}
\left(\inf_{\Omega} |v|^2\right)  \left\| \nabla \left(\frac{\partial_x \psi}{v(y)}\right)\right\|_{L^2}^2&\leq \int  |v(y)|^2 \left|\nabla \left(\frac{\partial_x \psi}{v(y)}\right)\right|^2 \rmd x\rmd y  \\
&=  \int  \frac{ \partial_x \psi }{v(y)} u \cdot \nabla {\omega}  \rmd x\rmd y = -\int {\omega}\nabla^\perp \psi\cdot  \nabla \left(\frac{\partial_x \psi}{v(y)}\right) \rmd x\rmd y \\
    &= \int {\omega} \left(\frac{\partial_x^2 \psi \partial_y \psi   }{v(y)}\right) \rmd x\rmd y -  \int v(y) {\omega}  \partial_y\left(\frac{\partial_x \psi}{v(y)}\right)  \left(\frac{\partial_x \psi}{v(y)}\right) \rmd x\rmd y .
\end{align*}
Note first that there exists a universal constant $C>0$ (depending only on $v$) such that
\begin{align*}
\left|  \int v(y) {\omega}  \partial_y\left(\frac{\partial_x \psi}{v(y)}\right)  \left(\frac{\partial_x \psi}{v(y)}\right) \rmd x\rmd y \right| 
&\leq C\|v\|_{L^\infty}\|\omega\|_{L^{2+}}   \left\| \nabla \left(\frac{\partial_x \psi}{v(y)}\right)\right\|_{L^2}^2.
\end{align*}
Now note that 
\begin{align*}
 \int {\omega} \left(\frac{\partial_x^2 \psi \partial_y \psi   }{v(y)}\right)\rmd x\rmd y &=  \int v(y) \left|\partial_x \left(\frac{\partial_x \psi}{v(y)}\right)\right|^2 \partial_y \psi  \rmd x\rmd y  +   \int \frac{\partial_x^2 \psi\partial_y^2 \psi \partial_y \psi   }{v(y)} \rmd x\rmd y \\
 &=  \int v(y) \left|\partial_x \left(\frac{\partial_x \psi}{v(y)}\right)\right|^2 \partial_y \psi  \rmd x\rmd y  +  \frac{1}{2} \int \partial_y \left(\frac{\partial_x \psi }{v(y)}\right) \partial_y\partial_x \psi  \partial_y \psi \rmd x\rmd y \\
  &=  \int v(y) \left|\partial_x \left(\frac{\partial_x \psi}{v(y)}\right)\right|^2 \partial_y \psi \rmd x\rmd y +  \frac{1}{2} \int  v(y)\left|\partial_y  \left(\frac{\partial_x \psi }{v(y)}\right)\right|^2   \partial_y \psi \rmd x\rmd y  \\
  &\qquad \qquad - \int v'(y) \partial_y \left(\frac{\partial_x \psi }{v(y)}\right) \frac{\partial_x \psi}{v(y)}   \partial_y \psi \rmd x\rmd y .
\end{align*}
Whence, for some universal $C>0$, we have
\begin{align*}
\left|  \int {\omega} \left(\frac{\partial_x^2 \psi \partial_y \psi   }{v(y)}\right) \right| &\leq C\|v\|_{C^1}\|u\|_{L^{\infty}}   \left\| \nabla \left(\frac{\partial_x \psi}{v(y)}\right)\right\|_{L^2}^2.
\end{align*}
\end{proof}
We remark that by  the result of Hamel and Nadirashvili  \cite{hamel2017shear, hamel2019liouville, HM23} can already be used to say that sufficiently regular velocities within an $L^\infty$ neighborhood of a non-vanishing shear are themselves shear flows.  The interest in Theorem \ref{rigiditynonstagnant} is the brevity of the proof, and that it applies to much lower regularity equilibria. 
Of course, Theorem \ref{rigiditynonstagnant} does not prevent nearby time periodic solutions with islands floating by with the current.

\section{Inviscid dynamical structures near shear flows: Proof of Theorem \ref{flexthm}}\label{flexsec}

We recap by recalling that Theorem \ref{rigiditynonstagnant} above, and also a Hamel and Nadirashvili's result, shows that non-vanishing shear flows are isolated from non-shears in $L^\infty$ of velocity. For Poiseuille flow  $v(y) = y^2$, we proved in Theorem \ref{mainshearrigid} that it is isolated in the $C^{2}$ topology from non-shear flows. We will now show that this rigidity is sharp by proving that non-shear flows appear arbitrarily close to $(y^n,0)$ in the $C^{n-}$ topology. 

As usual, for $n\geq 0$ and $0<\alpha<1$ we define for $v=v(y):(-1,1)\rightarrow \R$, 
\be 
\quad \Vert v \Vert_{C^{n,\alpha}(-1,1)} := \Vert v \Vert_{C^n(-1,1)} + | \partial_y^n v |_{C^\alpha(-1,1)},
\ee
where
\be
\Vert v \Vert_{C^n(-1,1)} := \sum_{k=0}^n \Vert \partial_y^k v \Vert_{L^\infty(-1,1)}, \quad | v |_{C^\alpha(-1,1)} := \sup_{y_1\neq y_2}\frac{|v(y_1)- v(y_2)|}{|y_1-y_2|^\alpha}.
\ee

The proof of the flexibility in spaces $C^{n-}=C^{n-1,\alpha}$, for all $0<\alpha<1$ is carried out in two steps. Firstly, we approximate the shear flow $v(y) = y^n$ by a zero velocity field in a small region. Then, we embed a compactly supported radial velocity field inside the approximating region. Since the background velocity is zero, this radial field constitutes a steady nonshear solution to Euler.

Let $\chi(\eta)$ be a smooth cut-off function such that 
\be
\chi(\eta)=
\begin{cases}
0, & |\eta|\leq 1, \\
1, & |\eta|\geq 2.
\end{cases}
\ee
and for $\varepsilon>0$ consider the shear flow approximation $v_\varepsilon(y) = v(y)\chi\left(\frac{y}{\varepsilon}\right)$. The purpose of the following lemma is to show that $v(y)$ is approximated by $v_\varepsilon(y)$ in $C^{n-}$. 
\begin{lemma}\label{lemma:Cnalpha approx}
Let $\alpha\in(0,1)$ and $\varepsilon_0>0$. Then, there exists $\varepsilon_1>0$ such that for all $\varepsilon\in(0,\varepsilon_1)$, 
\be
\Vert v - v_\varepsilon \Vert_{C^{n-1,\alpha}(-1,1)} < \frac12 \varepsilon_0.
\ee
\end{lemma}
\begin{proof}
For all $0\leq k \leq n-1$, a direct computation shows that
\begin{align*}
\Vert \partial_y^k(v - v_\varepsilon) \Vert_{L^\infty} &\leq \sum_{i=0}^{k} {k \choose i} \left\Vert \partial_y^{k-i}v(y)\partial_y^i\left( 1 - \chi\left(\frac{y}{\varepsilon} \right) \right) \right\Vert_{L^{\infty}} \\
&\lesssim_k \left\Vert y^{n-k}\left( 1 - \chi\left(\frac{y}{\varepsilon} \right)\right)\right\Vert_{L^\infty} +  \sum_{i=1}^{k}  \varepsilon^{-i}\left\Vert y^{n-k+i}\chi^{(i)}\left(\frac{y}{\varepsilon}  \right) \right\Vert_{L^{\infty}}
\end{align*}
Recalling that $\chi(\eta)=1$ for $|\eta|\geq 2 $ and $\chi^{(i)}(\eta) = 0$ for $|\eta|\not\in(1,2)$ we conclude that 
\be
\Vert \partial_y^k(v - v_\varepsilon) \Vert_{L^\infty} \leq C_k \varepsilon^{n-k},
\ee
for some $C_k>0$. On the other hand, we have that
\be
\left| \partial_y^{n-1}(v - v_\varepsilon) \right|_{C^\alpha}\lesssim_n \left| y\left( 1 - \chi\left(\frac{y}{\varepsilon}\right) \right) \right|_{C^\alpha} + \sum_{i=1}^{n-1}\left| y^{1+i}\varepsilon^{-i}\chi^{(i)}\left(\frac{y}{\varepsilon} \right) \right|_{C^\alpha}, 
\ee
where further, together with the properties of the smooth cut-off $\chi$, we have that 
\begin{align*}
\left| y\left( 1 - \chi\left(\frac{y}{\varepsilon}\right) \right) \right|_{C^\alpha} &= \varepsilon \left| \frac{y}{\varepsilon}\left( 1 - \chi\left(\frac{y}{\varepsilon}\right) \right) \right|_{C^\alpha} = \varepsilon^{1-\alpha}\left| \eta\left( 1 - \chi\left(\eta\right) \right) \right|_{C^\alpha} \lesssim \varepsilon^{1-\alpha}, 
\end{align*}
and similarly
\be
\sum_{i=1}^{n-1}\left| y^{1+i}\varepsilon^{-i}\chi^{(i)}\left(\frac{y}{\varepsilon} \right) \right|_{C^\alpha} = \varepsilon \sum_{i=1}^{n-1}\left| \left( \frac{y}{\varepsilon}\right)^{i+1}\chi^{(i)}\left(\frac{y}{\varepsilon} \right) \right|_{C^\alpha} \lesssim_n \varepsilon^{1-\alpha}.
\ee
The lemma follows from the definition of the $C^{n,\alpha}$ norm and choosing $\varepsilon_1$ sufficiently small. 
\end{proof}
Now that we have created a region in space with zero background velocity, we embed there a compactly supported radial velocity field, which is given by a suitably scaled smooth radial profile $\mathbf{U}:[0,1]\rightarrow \R$ such that $\mathbf{U}(r)=0$, for $r>\tfrac34$, all odd derivatives of $r\mathbf{U}(r)$ vanish at 0 and $\Vert (r^{-1}\partial_r)^{k}\mathbf{U}(r)\Vert_{L^\infty}< +\infty$, for all $0\leq k \leq n$. For convenience, we denote by
\begin{equation}
G(r) := \int_0^r s\mathbf{U}(s)\rmd s
\end{equation}
the stream function associated to the radial velocity field $(-y\mathbf{U}(r), x \mathbf{U}(r))$.
\begin{lemma}\label{lemma:radialstream}
Let $\mathbf{x}=(x,y)$ and $\Psi_\ve(\mathbf{x}) = A\ve^{n+2}F\left(\frac{\mathbf{x}}{\ve}\right)$, where $F(\mathbf{x}) = G(|\mathbf{x}|)$. Then, 
\begin{equation}
\Vert \Psi_\ve \Vert_{C^{n+1}(B_1)}\leq A\ve \Vert F \Vert_{C^{n+1}(B_\ve)}< \infty
\end{equation}
and the velocity field $w_\ve(\mathbf{x})=\nabla^\perp\Psi_{\ve}(\mathbf{x})$ is a radially smooth, compactly supported steady configuration of 2d Euler such that $\Vert w_\ve \Vert_{C^n}\lesssim \ve$.
\end{lemma}
\begin{proof}
Let us begin by noting that $G'(r)=r\mathbf{U}(r)$. Then,
\begin{equation}\label{eq:def grad F}
\partial_x F(\mathbf{x}) = \frac{G'(|\mathbf{x}|)}{|\mathbf{x}|}x=\mathbf{U}(|\mathbf{x}|)x, \quad \partial_y  F(\mathbf{x}) = \frac{G'(|\mathbf{x}|)}{|\mathbf{x}|}y = \mathbf{U}(|\mathbf{x}|)y.
\end{equation}
Since $\mathbf{U}(\cdot)$ is compactly supported in $[0,1]$, a repeated iteration shows that for all $1\leq k \leq n+1 $ and all $i+j=k$, 
\begin{equation}
|\partial_x^i \partial_y^{j}F(\mathbf{x})| \lesssim \sum_{l=0}^{k-1} \left| \left(\frac{1}{r}\partial_r\right)^{l}\mathbf{U}(r)\right| < +\infty,
\end{equation}
where we denoted $r=|\mathbf{x}|$. Hence, $\Vert F \Vert_{C^{n+1}}< \infty$ and \eqref{eq:def grad F} shows that 
\begin{equation}
w_\ve(\mathbf{x})=\nabla^\perp\Psi_\ve(\mathbf{x})=\mathbf{U}\left(\frac{|\mathbf{x}|}{\ve}\right)(-y,x)
\end{equation} 
is a smooth radial velocity field with compact support and $\Vert w_\ve \Vert_{C^n}\lesssim \ve$.
\end{proof}
We are now in position to prove the flexibility of the shear flow.
\begin{proof}[Proof of Flexibility]
For $0< \ve < \frac14$, we construct the velocity field 
\be
u_\varepsilon(\mathbf{x}) = (v_\varepsilon(y),0) + w_\varepsilon(\mathbf{x})
\ee
and we readily observe that thanks to the compact support of $\mathbf{U}$, we that for all $|\mathbf{x}|>\frac34\ve$, $w_\ve(\mathbf{x})\equiv 0$ and thus $u_\varepsilon(\mathbf{x}) = (v_\varepsilon(y),0)$. We shall show that $u_\ve$ is a stationary solution of Euler in $\mathbb{T}\times[-1,1]$. The slip boundary conditions at $y=\pm 1$ and the periodic boundary conditions follow from the compact support of $w_\ve$. As for the stationarity of the flow, since $v_\ve(y)\equiv 0$, for all $|y|\leq \ve$ and $w_\ve(\mathbf{x})=0$ for all $|\mathbf{x}|>\frac34\ve$, the two velocity field do not interact at all, and each of them is stationary under the Euler equations ($(v_\ve(y),0)$ is a shear flow and $w_\ve(\mathbf{x})$ is a radial flow) and so we conclude that $u_\varepsilon$ is a steady solution of the Euler equations. Finally, it remains to check $u_\varepsilon$ is close to the shear flow $u_*(x,y)=(y^n,0)$. Indeed, from Lemma \ref{lemma:Cnalpha approx} and Lemma \ref{lemma:radialstream} we have
\be
\Vert u_* - u_\varepsilon \Vert_{C^{n-1,\alpha}} = \Vert y^n -v_\varepsilon(y) \Vert_{C^{n-1,\alpha}} + \Vert w_\varepsilon(x,y) \Vert_{C^{n-1,\alpha}} < \varepsilon_0,
\ee
for $\varepsilon>0$ sufficiently small. The proof is concluded.
\end{proof}

The above construction can be generalized to admit shear flows that are not power laws, but locally behave like so. In particular, consider velocity profiles $u_*(x,y)=(v(y),0)$ such that $v\in C^{n,1}$ and assume there exists $y_0\in(-1,1)$ such that 
\begin{equation}\label{eq:cond v}\tag{V}
\sum_{i=0}^{n-1} |v^{(i)}(y_0)|=0, \quad v^{(n)}(y_0)\neq 0.
\end{equation}
To carry out the perturbation argument, we need good estimates on $v(y)$ around the critical stationary point $y_0$, we need to control the Taylor expansion of the shear flow $v(y)$ nearby $y_0$.
\begin{lemma}
Let $v(y)\in C^{n+1}$ satisfy \eqref{eq:cond v}. Then, for all $0\leq k \leq n-1$, there exists functions $h_k(y)\in {C}(-1,1)$ such that
\begin{equation}
v^{(k)}(y) = \frac{1}{(n-k)!}v^{(n)}(y_0)(y-y_0)^{n-k} + h_k(y)(y-y_0)^{n-k},
\end{equation}
where each $h_k\in C^1(-1,1)$ and is given by
\begin{equation}
h_k(y) :=\frac{1}{(y-y_0)^{n-k}}\int_{y_0}^{y}\int_{y_0}^{t_{k+1}}\ldots \int_{y_0}^{t_n} v^{(n+1)}(t_{n+1}) \rmd t_{n+1} \ldots \rmd t_{n-k}.
\end{equation}
\end{lemma}
\begin{proof}
Note that since $\partial_y^{k}v(y_0)=0$, for all $0\leq k \leq n-1$ and $\partial_y^n v(y_0) \neq 0$, we have
\begin{align}
v^{(k)}(y) &= \int_{y_0}^y v^{(k+1)}(t_{k+1}) \rmd t_{k+1} \\
&= \int_{y_0}^y\int_{y_0}^{t_{k+1}} \ldots\int_{y_0}^{t_{n-1}} v^{(n)}(t_{n}) \rmd t_{k+1}\ldots \rmd t_{n} \\
&=  \frac{v^{(n)}(y_0)}{(n-k)!}(y-y_0)^{n-k} + \int_{y_0}^{y}\int_{y_0}^{t_{k+1}}\ldots \int_{y_0}^{t_n} v^{(n+1)}(t_{n+1}) \rmd t_{n+1} \ldots \rmd t_{n-k}.
\end{align}
In particular, $v\in C^{n+1}([-1,1])$ shows that $h_k(y)$
is a $C^1$ function. The proof is concluded.
\end{proof}
Once the asymptotic properties of shear flow are established, we are in position to prove the flexibility of the steady state.
\begin{cor}\label{cor:genflex}
Let $u_*(x,y)=(v(y),0)$, with $v\in C^{n+1}$ and such that \eqref{eq:cond v} holds for some $y_0\in (-1,1)$. 
Then, there exist non-shear stationary states, time periodic and quasiperiodic of any number of non-commensurate frequencies $u(t)$ such that $\Vert u - u_* \Vert_{C^{n-}}< \varepsilon$.
\end{cor}

\begin{proof}
We shall again approximate the shear flow by the smooth cut-off $v_\varepsilon(y) = v(y)\chi\left(\frac{y-y_0}{\varepsilon}\right)$ and introduce a radial velocity field compactly supported on the small region where we approximate the shear flow by the zero profile. In particular, the result follows once we show that for all $0<\alpha< 1$ there holds
\be
\Vert v - v_\varepsilon \Vert_{C^{n-1,\alpha}(-1,1)} < \frac12 \varepsilon.
\ee
for all $\ve>0$ sufficiently small. Indeed, since
\begin{equation}
\Vert v - v_\varepsilon \Vert_{C^{n-1,\alpha}} = \sum_{i=0}^{n-1} \left\Vert \partial_y^i \left( v - v_\varepsilon\right) \right\Vert_{L^\infty} + \left\Vert \partial_y^{n-1} \left(v - v_\varepsilon \right) \right\Vert_{C^\alpha} 
\end{equation}
Now, denote $g(s)=1-\chi(s)$. For all $0\leq i \leq n-1$,
\begin{align*}
\Vert \partial_y^i(v - v_\varepsilon) \Vert_{L^\infty} &\leq \sum_{k=0}^{i} {i \choose k} \left\Vert v^{(k)}(y)\partial_y^{i-k}g\left(\tfrac{y-y_0}{\ve}\right)\right\Vert_{L^{\infty}} \\
&\lesssim \sum_{k=0}^{i} \ve^{-(i-k)}\left\Vert \left(v^{(n)}(y_0)(y-y_0)^{n-k} + h_k(y)(y-y_0)^{n-k}\right) g^{(i-k)}\left( \tfrac{y-y_0}{\ve}\right) \right\Vert_{L^\infty} \\
&\lesssim \Vert v \Vert_{C^{n+1}} \left\Vert \left(\frac{y-y_0}{\ve} \right)^{n-k}g^{(i-k)}\left( \frac{y-y_0}{\ve}\right) \right\Vert_{L^\infty} \ve^{n-i} \\
&\lesssim \ve^{n-i}.
\end{align*}
Similarly, we see that 
\begin{align*}
\Vert \partial_y^{n-1}(v - v_\varepsilon) \Vert_{C^\alpha} &\lesssim \sum_{k=0}^{n-1} \ve^{-(n-1-k)}\left\Vert \left(v^{(n)}(y_0)(y-y_0)^{n-k} + h_k(y)(y-y_0)^{n-k}\right) g^{(n-1-k)}\left( \tfrac{y-y_0}{\ve}\right) \right\Vert_{C^\alpha} \\
&\lesssim \Vert v \Vert_{C^{n+1}} \left\Vert \left(\frac{y-y_0}{\ve} \right)^{n-k}g^{(n-1-k)}\left( \frac{y-y_0}{\ve}\right) \right\Vert_{C^\alpha} \ve\lesssim \ve^{1-\alpha}.
\end{align*}
Here, we used  that $\Vert h_k \Vert_{C^\alpha}\lesssim \Vert h_k \Vert_{C^1}$, for all $0\leq \alpha< 1$ and the scaling $\Vert f(\lambda x) \Vert_{C^{\alpha}} = \lambda^\alpha \Vert f \Vert_{C^\alpha}$. 
\end{proof}

The above construction exploits the degeneracy of the velocity field; the perturbation is introduced in a small neighborhood of points where the background shear flow vanishes. In a sense, the location and nature of the perturbation is sharp, as can be seen from the following result, which is a straightforward consequence of  the result of Hamel and Nadirashvili  \cite{hamel2017shear, hamel2019liouville, HM23}.
\begin{cor}\label{cor:locshearrig}
Let $u_*(x,y) = (v(y),0)$ , with $v\in C^n$. Assume further that $v(0)=0$, $|v(y)|>0$, for all $y\neq 0$ and let $\delta>0$. Then, there exists $\ve_0>0$ such that any steady solution $u(x,y)$ of Euler for which $u(x,y)=u_*(x,y)$ in $(x,y)\in \mathbb{T}\times (-2\delta,2\delta)$ and 
\be
\Vert u - u_* \Vert_{L^\infty} < \ve_0 
\ee
must necessarily be a shear flow, that is $u(x,y)=(\tilde{v}(y),0)$.
\end{cor}
As will be clear from the proof, one may generalize the result to include background shear flow profiles $v(y)$ that vanish in a finite number of points (or even in whole subintervals of $(-1,1)$), by requiring the velocity perturbation to be compactly supported away from the zeroes of the profile.
\begin{proof}
Let 
$$m = \min_{y\in[-1,-\delta]\cup [\delta, 1]}|v(y)|>0$$ 
and choose $\ve_0=\frac{m}{2}$. Then, since $u(x,y) = u_*(x,y) = (v(y),0)$ for all $(x,y)\in \mathbb{T}\times (-2\delta,2\delta)$, we observe that the straight line 
$$\Gamma_{\delta}=\lbrace(x,\delta)\rbrace_{x\in\mathbb{T}}$$
is a fluid boundary for $u(x,y)$, since $u(x,y)|_{\Gamma_\delta} = (v(\delta),0)$, it is tangential to $\Gamma_\delta$. The same behavior occurs for $\Gamma_{-\delta}$ and, in particular, we can decompose our domain $M=\mathbb{T}\times[-1,1]$ and look at the Euler equations on each subdomain $M_+=\mathbb{T}\times[\delta,1]$ and $M_-=\mathbb{T}\times[-1,\delta]$. Moreover,
\begin{equation}
\inf_{M_+\cup M_-}|u| \geq \inf_{M_+\cup M_-}|u_*| - \ve_0 = \frac{m}{2}>0,
\end{equation}
so that, in particular, the infimum on each $M_+$ and $M_-$ is strictly positive. Then, Theorem 1.1 of \cite{CDG21} applied to each $M_+$ and $M_-$ separately ensures that $u(x,y)$ is a shear flow on each $M_+$ and $M_-$. The proof is concluded.
\end{proof}
Some remarks on limitations and its generalizations are in order.

\begin{remark}[Sharpness]
{In the case of Couette flow ($n=1$, constant vorticity), the rigidity/flexibility with Sobolev topology established by \cite{LinZeng} is sharp and their H\"{o}lder regularity is better than $C^{1-}$ that we prove (it is flexible in $C^{1, 1/2-}$). We also remark that for shear flows that are not pure power laws, the regularity threshold can be increased. Indeed, there exists a family of strictly monotone shear flows for which our $C^{1-}$ threshold can be improved at least half a derivative more, see \cite{sinambela2023transition}. On the other hand, the H\"{o}lder regularity threshold $C^{1,\frac12-}$ of Poiseuille in \cite{CastroLear24b} falls short of the $C^{2-}$ threshold that we establish here and which we show to be sharp. Given the previous results, we expect the $C^{n-}$ flexibility to be sharp for pure power laws $v(y)=y^n$, at least for even $n\geq 2$. Namely, we suspect that, as in the Poiseuille case, there exists an $\ve:=\ve(u_*, M)$ such that all stationary solutions $u$ satisfying $\|u-(y^n,0)\|_{C^n(M)}<\ve$ are shear flows.}
\end{remark}	
	
\begin{remark}[Structural Instability]
The non-shear steady solutions that we present are not generic, in the sense that the stream-functions of such velocity fields are not Morse functions, they are expected to be unstable. This is to be compared with the works of \cite{CastroLear23, CastroLear24b, CastroLear24, LinZeng} where non-trivial Euler steady solutions are constructed. These are stable under perturbation, and exhibit truly two-dimensional features. 
\end{remark}	
	
\begin{remark}[Coliving Vortices]
The radial velocity field introduced in Lemma \ref{lemma:radialstream} is localized on the origin and its compact support is small, it is included in the ball of radius $\frac34\ve$. Thus, one may introduce additional radial velocity fields, with different velocity profiles $\mathbf{U}$ whose compact supports do not intersect each other and are included in the open band $M_\ve=\lbrace (x,y)\in M\,:\, |y| < \frac{9}{10}\ve\rbrace$. The resulting velocity field is then  a non-shear steady solution of Euler close to the background shear flow and composed of several vortices with different profiles and strengths living in the same stationary band.
\end{remark}

\begin{remark}[Travelling Wave Vortices]
Even though Corollary \ref{cor:locshearrig} shows that non-shear steady state perturbations must be non-trivial (non-vanishing) in regions close to the stationary points of the background shear flow, we may also find comoving structures (such as time-periodic traveling waves) concentrated away from these stationary points of the shear flow profile $v(y)$. Indeed, for any $y_0\in(-1,1)$ such that $v_0:=v(y_0)\neq 0$, the shear flow $\overline{v}(y)=v(y)-v_0$ vanishes at $y=y_0$ and thus, thanks to Corollary \ref{cor:genflex}, we can embed a small radial velocity field $w_\ve(x,y-y_0)$ resulting in the non-shear steady state $\overline{u}(x,y)=\chi(\frac{y-y_0}{\ve})(\overline{v}(y),0) + w_\ve(x,y-y_0)$. Hence, the traveling wave
\begin{equation}
u_\ve(t,x,y) = \overline{u}_\ve(x-v(y_0)t,y) + (v(y_0),0)
\end{equation}
is a solution to the Euler equations and it is $\ve>0$ close to the shear steady state $u_*=(v(y),0)$ in $C^{n,\alpha}(M)$ regularity, for all $0\leq\alpha<1$ and $n$ counting the number of derivatives of $v(y)$ vanishing at $y_0$. In particular, we can have a family of coliving vortices embedded now on a traveling wave band of speed $v(y_0)$. The method can also be generalized to allow more traveling wave bands, each of them with its own family of coliving vortices and its own wave speed. For example, for any $y_1\in(-1,1)$ sufficiently far from $y_0$, we may define
\begin{equation}
u_\ve^1(t,x,y) = \chi\left( \frac{y-y_1}{\ve}\right) \left( u_\ve(t,x-v(y_1)t,y)-(v(y_1),0)\right) + w^1_\ve(x-v(y_1)t,y-y_1) + (v(y_1),0),
\end{equation}
where now $u_\ve^1$ is a travelling wave solution of Euler of speed $v(y_1)$ with two travelling wave bands, one of them with relative wave speed $v(y_0)+v(y_1)$.
\end{remark}

\begin{remark}[Embedded Vortices]
For profiles $\mathbf{U}$ such that $\mathbf{U}(r)=1$, for all $\tfrac13 < r < \tfrac23$, the velocity field $u_\varepsilon=w_\varepsilon=A\epsilon^n r\mathrm{e}_\theta$ is simply a solid body rotation in that annular region. Hence, up to a rigid rotation of speed $A\ve^n$ to adjust the solid body motion, we can embed a family of coliving smaller vortices within the annular region. The overall resulting non-shear velocity field is no longer a steady but a periodic or quasiperiodic solution of Euler, depending on the angular velocities of all the solid body motions involved. One may even choose to embed any of these smaller vortices on any (or several) members of the family of coliving vortices sitting on the traveling wave bands discussed above.
\end{remark}

\begin{remark}[Locally non-radial Vortices]
{Instead of the radial velocity fields that we introduce in Lemma \ref{lemma:radialstream}, we can opt for placing there non-locally radial and $C^n$ regular compactly supported stationary Euler flows, whose existence was recently obtained in \cite{enciso2024smooth}. As such, the steady states that we find in Theorem \ref{flexthm} need not have any local symmetry.}
\end{remark}

To conclude, the above remarks show the dynamical richness and endless configurations of solutions to Euler nearby arbitrary shear flows in low regularity neighborhoods.

 \subsection*{Acknowledgments}   We thank T. M. Elgindi and F. Torres de Lizaur   for insightful remarks. We particular acknowledge T. M. Elgindi with whom we obtained the result on inviscid dynamical structures nearby shears (Theorem \ref{flexthm}) in collaboration.
The research of TDD was partially supported by the NSF DMS-2106233 grant, NSF CAREER award \#2235395, a Stony Brook University Trustee's award as well as an Alfred P. Sloan Fellowship. The research of MN was partially supported by the Doris Chen Mobility Award.

\bibliographystyle{abbrv}
\bibliography{refs}

\end{document}